\documentclass[preprint]{elsarticle}

\usepackage[utf8]{inputenc}
\usepackage[T1]{fontenc}

\usepackage{adjustbox, amsfonts, amsmath, amssymb, amsthm, color, enumerate, enumitem, tikz-cd}

\usepackage{color}
\definecolor{myurlcolor}{rgb}{0.1,0.1,0.8}
\definecolor{mylinkcolor}{rgb}{0.05,0.05,0.4}

\usepackage{hyperref}
\hypersetup{colorlinks,
	linkcolor=mylinkcolor,
	citecolor=mylinkcolor,
	urlcolor=myurlcolor}

\usepackage{cleveref}

\usepackage{tocloft}
\newtheoremstyle{dotless_thm}{}{}{\itshape}{}{\bfseries}{}{ }{}
\newtheoremstyle{dotless_def}{}{}{}{}{\bfseries}{}{ }{}
\newtheoremstyle{dotless_blue}{}{}{\color{blue}}{}{\bfseries}{}{ }{}
\newtheoremstyle{dotless_red}{}{}{\color{red}}{}{\bfseries}{}{ }{}

\theoremstyle{dotless_thm}
\newtheorem{thm}{Theorem}
\newtheorem*{thm*}{Theorem}
\numberwithin{thm}{section}
\newtheorem{cor}[thm]{Corollary}
\newtheorem*{cor*}{Corollary}
\newtheorem{lem}[thm]{Lemma}
\newtheorem*{lem*}{Lemma}
\newtheorem{prop}[thm]{Proposition}
\newtheorem*{prop*}{Proposition}

\newtheorem*{thmMH_normed_gps}{Theorem~\ref{thm:MH_normed_gps}}
\newtheorem*{thmMH_suspension}{Theorem~\ref{thm:MH_suspension}}

\theoremstyle{dotless_def}
\newtheorem{defn}[thm]{Definition}
\newtheorem{eg}[thm]{Example}
\newtheorem{rmk}[thm]{Remark}

\newcommand{\cn}{\mathbf}
\newcommand{\cf}{\mathcal}
\newcommand{\cl}{\mathbf}
\newcommand{\bb}{\mathbb}
\newcommand{\bdiag}{\begin{equation*}\begin{tikzcd}}
\newcommand{\ediag}{\end{tikzcd}\end{equation*}}
\newcommand{\xto}{\xrightarrow}
\newcommand{\st}{\, \big| \,}
\newcommand{\ob}{\mathrm{ob}}
\newcommand{\op}{\mathrm{op}}

\newcommand{\arr}{\mathrm{arr}}
\newcommand{\id}{\mathrm{Id}}

\newcommand{\Ch}{\mathrm{Ch}}
\newcommand{\ost}{\circledast}
\newcommand{\im}{\mathrm{im}}
\newcommand{\diag}{\mathrm{diag}}
\newcommand{\Tot}{\mathrm{Tot}}
\newcommand{\Tor}{\mathrm{Tor}}

\newcommand{\pw}{\mathrm{pw}}
\newcommand{\ab}{\mathbb}

\renewcommand{\th}{\textrm{th}}
\renewcommand{\vec}{\mathbf}

\newsavebox{\OneSingleA}
\newsavebox{\OneSingleB}
\newsavebox{\TwoSinglesA}
\newsavebox{\TwoSinglesB}
\newsavebox{\OnePairA}
\newsavebox{\OnePairB}
\newsavebox{\OnePairC}
\newsavebox{\TwoPairs}
\newsavebox{\OneTriples}
\newsavebox{\TwoTriples}

\newsavebox{\MOneSingleA}
\newsavebox{\MOneSingleB}
\newsavebox{\MTwoSinglesA}
\newsavebox{\MTwoSinglesB}
\newsavebox{\MOnePairA}
\newsavebox{\MOnePairB}
\newsavebox{\MOnePairC}
\newsavebox{\MTwoPairs}
\newsavebox{\MOneTriples}
\newsavebox{\MTwoTriples}
\newsavebox{\MOneDoubleA}
\newsavebox{\MTwoDoubles}

\title{Iterated magnitude homology}

\date{}

\begin{document}

\begin{frontmatter}

\author[1]{Emily Roff}
\ead{emily.roff@ed.ac.uk}
\address[1]{School of Mathematics and Maxwell Institute for Mathematical Sciences,\\University of Edinburgh, James Clark Maxwell Building, Edinburgh EH9 3FD}

\begin{abstract}
Magnitude homology is an invariant of enriched categories which generalizes ordinary categorical homology---the homology of the classifying space of a small category. The classifying space can also be generalized in a different direction: it extends from categories to bicategories as the geometric realization of the geometric nerve. This paper introduces a hybrid of the two ideas: an \emph{iterated magnitude homology} theory for categories with a second- or higher-order enrichment. This encompasses, for example, groups equipped with extra structure such as a partial ordering or a bi-invariant metric. In the case of a strict 2-category, iterated magnitude homology recovers the homology of the classifying space; we investigate its content and behaviour when interpreted for partially ordered groups, normed groups, and strict \(n\)-categories for \(n > 2\).
\end{abstract}
\end{frontmatter}


\tableofcontents


\section{Introduction}\label{sec:intro}

The homology of a small category---meaning the homology of its classifying space---is a well-established construction which unifies various even more classical theories. Specialized to groups, it recovers group homology; specialized to posets, it recovers the homology of the order complex. The homology of small categories can, in turn, be generalized in two quite different directions. On the one hand, it can be lifted to a higher categorical dimension: the classifying space functor can be extended along the inclusion of categories into bicategories as the geometric realization of the \emph{geometric nerve}. On the other hand, the theory can be enriched. \emph{Magnitude homology} extends the homology of small categories---which are categories enriched in sets---to categories enriched in some more exotic monoidal category \(\cf{V}\).

The scheme for constructing magnitude homology runs as follows. First, one must choose a functor \(\Sigma\) from \(\cf{V}\) to some abelian category \(\ab{A}\). (We will refer to this as a choice of coefficients; conventionally, \(\Sigma\) is thought of as recording information about the `size' of objects in \(\cf{V}\).) The category \(\ab{A}\) should carry a closed monoidal structure, and \(\Sigma\) should be strong symmetric monoidal. There is then a recipe to lift \(\Sigma\) to a functor on the category \(\cf{V}\cn{Cat}\) of \(\cf{V}\)-categories, valued in the category of simplicial objects in \(\ab{A}\); this is the \emph{magnitude nerve} functor. The magnitude homology of a \(\cf{V}\)-category is defined to be the homology of the Moore complex associated to its magnitude nerve. (This terminology is explained and more details provided in Sections \ref{sec:enriched_cats} and \ref{sec:MH}, below.)

Compared to the geometric nerve and the classifying space of a bicategory---which date back to work of Duskin and Street in the 1980s \cite{StreetAlgebra1987}---magnitude homology is a young theory. First defined for graphs by Hepworth and Willerton in 2017 \cite{HepworthWillerton2017}, it was extended shortly thereafter by Leinster and Shulman \cite{LeinsterMagnitude2017v3} to encompass any category with a semicartesian enrichment, meaning that the monoidal unit in \(\cf{V}\) is terminal. Despite its relatively recent emergence, the theory has already shown itself to be productive. In the \(\cn{Set}\)-enriched setting it recovers ordinary categorical homology: if we take \(\Sigma\) to be the free abelian group functor on \(\cn{Set}\), then the magnitude nerve of a small category is the simplicial abelian group freely generated by the ordinary nerve, so the magnitude homology of a small category is the homology of its classifying space. In other contexts, however, it offers something altogether new.

It has proved to be an especially rich invariant of metric spaces, regarded as categories enriched in the poset of nonnegative real numbers. With a suitable choice of coefficient functor, the magnitude homology in low degrees of a metric space captures specifically geometric information about convexity, curvature and the existence and uniqueness of geodesics (e.g.~\cite{AsaoMH2019,GomiGeodesic2019,KanetaYoshinaga2018}). In the context of directed graphs, Asao has established a close connection between magnitude homology and Grigor'yan--Muranov--Yau \emph{path homology}---appearing, respectively, on the first and second pages of a certain spectral sequence---and this relationship has already been exploited to prove new results about path homology and cast light on the homotopy theory of graphs \cite{AsaoMH2022, CaputiCollari2023, CarranzaEtAl2022}. More recently still, Tajima and Yoshinaga have made use of causal posets to define a notion of `magnitude homotopy type' for metric spaces \cite{TajimaYoshinaga2023}.

This grounding in metric geometry and graph theory gives the applications of magnitude homology to date  quite a different flavour from the abstract homotopy-theoretic or higher category-theoretic settings in which one typically encounters the bicategorical classifying space. Yet, of course, some bicategories are also enriched categories: a strict 2-category is precisely a category enriched in \(\cn{Cat}\), the category of small categories. It makes sense, then, to ask whether the classifying space of a strict 2-category can be related to its magnitude homology with respect to some choice of coefficient functor on \(\cn{Cat}\).

To answer that question, it is helpful to know that the classifying space of a strict 2-category has an alternative construction, essentially due to Segal \cite{SegalClassifying1968} and given an especially convenient description by Bullejos and Cegarra in \cite{BullejosCegarra2003}. Instead of forming the geometric nerve of a strict 2-category \(\cl{X}\)---which is a simplicial set---the idea is to build a \emph{bi}simplicial set by iterating the ordinary nerve construction for small categories: applying the nerve functor first locally (to each hom-category in \(\cl{X}\)) and then globally, to the resulting simplicial category. Bullejos and Cegarra prove that the diagonal of this bisimplicial set is naturally homotopy equivalent, after geometric realization, to the geometric nerve (\cite{BullejosCegarra2003}, Theorem 1).

This paper takes the bisimplicial approach as inspiration to construct a modified `magnitude homology' theory for categories with a second- or higher-order enrichment---categories enriched in some category of \(\cf{V}\)-categories---where a coefficient functor on the base \(\cf{V}\) has been chosen. In the framework of magnitude homology, the analogue of Segal's approach is to treat the magnitude nerve itself as a coefficient functor on \(\cf{V}\cn{Cat}\). This produces, for each \(\cf{V}\cn{Cat}\)-enriched category \(\cl{X}\), a bisimplicial object in the abelian category \(\ab{A}\), whose diagonal we call the \emph{iterated magnitude nerve} of \(\cl{X}\). The \emph{iterated magnitude homology} of \(\cl{X}\) is defined to be the homology of the Moore complex associated to its iterated magnitude nerve. (Details are given in \Cref{sec:iterated_MH}.)

By design---and by Bullejos and Cegarra's Theorem 1---the iterated magnitude homology of a strict 2-category coincides with the topological homology of its classifying space; this is our \Cref{thm:MH_Duskin}. Because we have the freedom to vary the base category \(\cf{V}\) and the coefficient functor \(\Sigma\), we also have the freedom to extend the theory to new settings.

As it turns out, higher-order enrichment occurs in nature more often than one might imagine. In particular, various commonly-encountered structures on groups can be described as second-order enrichments in familiar base categories. Any congruence on a group (equivalently, any normal subgroup) gives rise to an enrichment in \(\cn{Cat}\) and thus a second-order enrichment in \(\cn{Set}\). A group is called `partially ordered' not simply when its underlying set is partially ordered, but rather when it has an enrichment in the category \(\cn{Poset}\) of posets and monotone maps, and thus a second-order enrichment in the category of Boolean truth values. And a group equipped with a norm that is constant on conjugacy classes---equivalently, a translation-invariant metric---is precisely one with an enrichment in the category \(\cn{Met}\) of metric spaces and 1-Lipschitz maps, and thus a second-order enrichment in the poset of nonnegative reals. (These examples are explained in detail in \Cref{sec:enriched_groups}.) 

To learn more about a partially ordered group or a group with a conjugation-invariant norm, one could choose to analyse its ordinary group homology, or alternatively to analyse the homology of its order complex or its magnitude homology as a metric space. The first option means neglecting the partial ordering or the metric; the others mean neglecting the group operation. In order to keep both forms of structure in view, one might prefer instead to take the magnitude homology of the group as a one-object \(\cn{Poset}\)- or \(\cn{Met}\)-enriched category---this is where the iterated theory comes in.

We investigate iterated magnitude homology for each of these examples. We begin by considering the special case of a strict 2-group, for which the well-known correspondence with homotopy 2-types makes the analysis particularly straightforward, and then extend that analysis to any \(\cn{Cat}\)-enriched group. Applied to a partially ordered group, we learn that iterated magnitude homology is a rather insensitive invariant, failing to distinguish (at least in its first two groups) between a partial order and the equivalence relation it generates. In the case of a normed group, though, things become more interesting. Here we find that the iterated theory captures both topological and metric features, recovering ordinary group homology and at the same time bearing information about the geometry of the group under the norm.

Finally we take the iteration further, constructing an iterated magnitude homology theory for strict \(n\)-categories for \(n \geq 2\). For a certain special class of examples we are able to describe the homology completely, proving a theorem formally analogous to the suspension theorem of classical algebraic topology.

\paragraph{The structure of the paper}

To fix conventions and notation, and to introduce the key examples, we begin in \Cref{sec:enriched_cats} by reviewing the essential concepts of enriched category theory, before describing in \Cref{sec:MH} the construction of magnitude homology and a few of its properties. \Cref{sec:MH} also contains what is in a sense the central technical fact of the paper. In the bisimplicial construction of the classifying space of a strict 2-category, it is crucial that the nerve functor on \(\cn{Cat}\) is strong symmetric monoidal: this property ensures that taking the nerve of every hom-category produces a simplicial category. Accordingly, before we can imitate that construction we must prove that the magnitude nerve is strong symmetric monoidal; this is \Cref{prop:MB_monoid}.

As a first application of the monoidal structure, we are able to prove---in \Cref{sec:Kunneth}---a completely general Eilenberg--Zilber-type theorem for the magnitude chain complex (\Cref{thm:EZ_magnitude}). From there we can derive two K\"unneth formulae (Theorems \ref{thm:MH_Kunneth} and \ref{thm:MH_Kunneth_metric}) extending those proved for the magnitude homology of undirected graphs by Hepworth and Willerton \cite{HepworthWillerton2017} and for classical metric spaces by Bottinelli and Kaiser \cite{BottinelliKaiser2021}. These results are not essential to the later parts of this paper, but are of independent practical value as they establish, in particular, that a K\"unneth formula holds for the magnitude homology of non-symmetric metric spaces such as directed graphs.

Sections \ref{sec:iterated_MH}--\ref{sec:MH_ncats} contain the main substance of the paper. In \Cref{sec:iterated_MH} we explain the construction of iterated magnitude homology, and the remaining sections are devoted to investigating the information it carries when interpreted for enriched groups and strict \(n\)-categories. Because the concept of an enriched group is not a commonplace one, we spend \Cref{sec:enriched_groups} introducing it and describing our three main classes of examples: strict 2-groups, partially ordered groups, and normed groups. In \Cref{sec:MH_groups} we analyse the iterated magnitude homology in low degrees for each of these, before turning in \Cref{sec:MH_ncats} to study the iterated magnitude homology of strict \(n\)-categories for \(n \geq 2\).

\paragraph{Acknowledgements}

I am grateful to Tom Leinster for valuable conversations during the development of this work, and to Richard Hepworth for his comments on the early draft of this paper which forms part of my PhD thesis.

\paragraph{Declaration of interest}

None.


\section{Enriched categories}\label{sec:enriched_cats}

The basic idea of enriched category theory is to describe category-like structures in which, instead of a {set} of morphisms between every pair of objects, one has an `object of morphisms' drawn from some other category of interest. For this to make sense, as we shall see, it is necessary that the `enriching' category should be \emph{monoidal}. A \textbf{monoidal category} is a category \(\cf{V}\) equipped with an associative binary operation \(\otimes\) on its objects which extends to a functor \(\cf{V}\times\cf{V} \to \cf{V}\), and an object \(I\) in \(\cf{V}\) which is a unit for \(\otimes\). These data must satisfy associativity and unit axioms up to coherent isomorphism; for details, see Chapter VII in \cite{MacLane1997}.

We will only consider monoidal structures that are \textbf{symmetric}, meaning there are isomorphisms \(\gamma_{u,v}: u \otimes v \to v \otimes u\), natural in \(u,v \in \ob(\cf{V})\) and satisfying appropriate coherence conditions. A symmetric monoidal category \((\cf{V}, \otimes, I)\) is said to be \textbf{semicartesian} if \(I\) is the terminal object in \(\cf{V}\), and \textbf{cartesian} if \(\otimes\) is the categorical product in \(\cf{V}\).

\begin{eg}
\begin{itemize}
\item The cartesian product of sets makes the category \(\cn{Set}\) of sets and functions into a cartesian monoidal category. The unit object for this monoidal structure is the one-element set: \(I = \{*\}\).
\item The category of Boolean truth values has two objects, \(T\) and \(F\), and a single non-identity arrow \(F \to T\). We will call this category \(\cn{Truth}\). The operation of conjunction---denoted by \(\wedge\)---is the categorical product here, making \(\cn{Truth}\) into a cartesian monoidal category with unit object \(T\).
\item Any preorder can be regarded as a category in which every hom-set has at most one element. In particular this is true of the poset \([0, \infty]\) of nonnegative real numbers, extended to infinity. The objects of the corresponding category are the elements of \([0, \infty]\), with an arrow \(u \to v\) if and only if \(u \geq v\). This category carries many monoidal structures, but in this paper we consider just one of them: the structure in which \(u \otimes v = u + v\). The unit for \(+\) is \(0\), which is terminal in \([0, \infty]\), so \(([0, \infty], +, 0)\) is semicartesian.
\end{itemize}
\end{eg}

Let \((\cf{V}, \otimes, I)\) be a monoidal category. A (small) \textbf{category enriched in \(\cf{V}\)}, or \textbf{\(\cf{V}\)-category}, consists of a set \(\ob(\cl{X})\) of objects and for each pair \(x,y \in \ob(\cl{X})\) an object \(\cl{X}(x,y)\) in \(\cf{V}\). For each \(x \in \ob(\cl{X})\) there must be a morphism \(\id_x: I \to \cl{X}(x,x)\) in \(\cf{V}\)---thought of as `selecting the identity' on \(x\)---and for each triple \(x,y,z\) there must be a morphism \(\circ: \cl{X}(y,z) \otimes \cl{X}(x,y) \to \cl{X}(x,z)\) in \(\cf{V}\), performing `composition'. These data must satisfy axioms mirroring those in the definition of an ordinary category; details can be found in \cite{KellyBasicConcepts1982}.

\begin{eg}
\begin{itemize}
\item A category enriched in \((\cn{Set}, \times, \{*\})\) is an ordinary small category. For instance, any group can be regarded as a \(\cn{Set}\)-category with a single object, in which every morphism is invertible.
\item A category enriched in \((\cn{Truth}, \wedge, T)\) is a set equipped with a reflexive, transitive relation---in other words, a preorder. In particular, every poset is a \(\cn{Truth}\)-category.
\item A category enriched in \(([0, \infty], +, 0)\) consists of a set \(X\) and, for every pair \(x,y\in X\), a number \(d(x,y) \in [0, \infty]\). The existence of `identity morphisms' here says that for every \(x \in X\) we have \(0 \geq d(x,x)\) and so \(d(x,x) = 0\); `composition' gives us, for every triple \(x,y,z \in X\), the inequality
\[d(y,z) + d(x,y) \geq d(x,z).\]
Thus, as was first observed by Lawvere \cite{LawvereMetric1974}, a category enriched in \(([0, \infty], +, 0)\) is a \textbf{generalized metric space}: one whose metric need not be separated or symmetric, and may take infinite values.

Any directed graph can be made into a generalized metric space by declaring the distance from a vertex \(u\) to a vertex \(v\) to be the length of the shortest directed path from \(u\) to \(v\), or infinity if no such path exists. This metric will always be separated but is typically \emph{not} symmetric nor finite-valued in general. Because we often have such examples in mind, we adopt the nonstandard convention that the term \textbf{metric space} means a generalized metric space whose metric is separated. We will use the term \textbf{classical metric space} to refer to a space whose metric is separated, symmetric and always finite-valued.
\end{itemize}
\end{eg}

We will sometimes refer to the \textbf{underlying (ordinary) category} of a \(\cf{V}\)-category \(\cl{X}\). This is the category \(\cl{X}_0\) whose objects are those of \(\cl{X}\), with hom-sets \(\cl{X}_0(x,y) = \cf{V}(I, \cl{X}(x,y))\). For instance, the underlying ordinary category of a metric space is the discrete category on its set of points.

There is a notion of \textbf{enriched functor} between \(\cf{V}\)-categories, and thus for every monoidal category \((\cf{V}, \otimes, I)\) there is a category of \(\cf{V}\)-categories, which we denote by \(\cf{V}\cn{Cat}\) \cite{KellyBasicConcepts1982}. If \((\cf{V}, \otimes, I)\) is symmetric monoidal, then \(\cf{V}\cn{Cat}\) also carries a symmetric monoidal structure: given two \(\cf{V}\)-categories \(\cl{X}\) and \(\cl{Y}\), their tensor product \(\cl{X} \otimes \cl{Y}\) has object set \(\ob(\cl{X}) \times \ob(\cl{Y})\), with hom objects
\[(\cl{X} \otimes \cl{Y}) ((x,y), (x',y')) = \cl{X}(x,x') \otimes \cl{Y}(y,y').\]
(The symmetry of the monoidal structure in \(\cf{V}\) is used to define composition in \(\cl{X} \otimes \cl{Y}\).) The unit for this monoidal structure is the one-object \(\cf{V}\)-category \(\cl{I}\) whose unique hom-object is the unit \(I\) in \(\cf{V}\).

\begin{eg}
\begin{itemize}
\item The monoidal category of categories enriched in \(\cn{Set}\) is just the ordinary category \(\cn{Cat}\) of small categories, with tensor product the categorical product and unit the terminal category.
\item A \(\cn{Truth}\)-functor between preorders is precisely a monotone map. The tensor product of preorders \(P\) and \(Q\) is the product set \(P \times Q\) equipped with the preorder in which \((p,q) \leq (p',q')\) if and only if \(p \leq p'\) and \(q \leq q'\). We will denote the monoidal category of preorders and monotone maps by \(\cn{PreOrd}\), and its full subcategory of posets by \(\cn{Poset}\).
\item An enriched functor between generalized metric spaces \((X,d_X)\) and \((Y,d_Y)\) is a function \(f: X \to Y\) such that \(d_Y(f(x),f(x')) \leq d_X(x,x')\) for all \(x,x' \in X\); in other words, a 1-Lipschitz map from \(X\) to \(Y\). Hereafter we will refer to these as \textbf{short maps}. The tensor product \(X \otimes Y\) is the product set \(X \times Y\) equipped with the \(\ell^1\) product metric:
\[d_{X \otimes Y}((x,y), (x',y')) = d_X(x,x') + d_Y(y,y').\]
We will denote the monoidal category of generalized metric spaces and short maps by \(\cn{GMet}\), and its full subcategory of metric spaces by \(\cn{Met}\).
\end{itemize}
\end{eg}

Given monoidal categories \((\cf{V}, \otimes, I_\cf{V})\) and \((\cf{W}, \bullet, I_\cf{W})\), a functor \(F: \cf{V} \to \cf{W}\) is \textbf{monoidal} if there are morphisms \(\phi: I_\cf{W} \to F(I_\cf{V})\) and \(\phi_{u,v}: F(u) \bullet F(v) \to F(u \otimes v)\), natural in \(u,v \in \ob(\cf{V})\) and satisfying certain coherence conditions. If these morphisms are isomorphisms, \(F\) is said to be \textbf{strong monoidal}. If the monoidal structures on \(\cf{V}\) and \(\cf{W}\) are symmetric then one can ask for \(F\) to be \textbf{strong symmetric monoidal}, meaning, of course, that it respects the symmetric structure: for all \(u,v \in \cf{V}\) we have \(\phi_{v,u} \circ \gamma_{F(u),F(v)} = F(\gamma_{u,v}) \circ \phi_{v,u}\).

\paragraph{Conventions} Throughout this paper, \((\cf{V}, \otimes, I)\) will always denote a {semicartesian} monoidal category, and \((\ab{A}, \otimes, \bb{I})\) will always denote an abelian category with a symmetric monoidal structure which is \textbf{closed}, meaning that for every object \(C\) in \(\ab{A}\) the functor \(- \otimes C: \ab{A} \to \ab{A}\) has a right adjoint. Since left adjoints preserve colimits, the closed structure on \(\ab{A}\) (along with the symmetry) guarantees that the tensor product preserves arbitrary coproducts in each variable.

We denote the simplex category by \(\Delta\). For any symmetric monoidal closed abelian category \(\ab{A}\), the category \([\Delta^\op, \ab{A}]\) of simplicial objects in \(\ab{A}\) is also abelian, and carries a symmetric monoidal closed structure inherited pointwise from \(\ab{A}\), which we denote by \(([\Delta^\op, \ab{A}], \otimes_\pw, I_{\pw})\).


\section{Magnitude homology}\label{sec:MH}

All definitions in this section are due to Leinster and Shulman \cite{LeinsterMagnitude2017v3} (following Hepworth and Willerton \cite{HepworthWillerton2017}). They define the \emph{magnitude homology} of a category with a semicartesian enrichment to be the homology of the Moore complex associated to a certain simplicial object, the \emph{magnitude nerve}.

\begin{defn}\label{def:mag_nerve}
Let \((\cf{V}, \otimes, I)\) be a {semicartesian} monoidal category, and \((\ab{A}, \otimes, \bb{I})\) a symmetric monoidal closed abelian category. Let \(\Sigma : \cf{V} \to \ab{A}\) be a strong symmetric monoidal functor. The \textbf{magnitude nerve} of  a \(\cf{V}\)-category \(\cl{X}\) with respect to \(\Sigma\) is a simplicial object \(MB^\Sigma(\cl{X})\) in \(\ab{A}\) given for each \(n \in \bb{N}\) by
\begin{equation*}
MB_n^\Sigma(\cl{X}) = \bigoplus_{x_0,\ldots,x_n \in \cl{X}} \Sigma\cl{X}(x_0,x_1) \otimes \cdots \otimes \Sigma\cl{X}(x_{n-1},x_n).
\end{equation*}
In particular, since the empty tensor product is the unit object, we have \(MB_0^\Sigma(\cl{X}) = \bigoplus_{x \in \cl{X}} \mathbb{I}\).
The simplicial structure is defined as follows:
\begin{itemize}
\item For \(1 \leq i \leq n-1\) the face map \(\delta_n^i: MB_n^\Sigma(\cl{X}) \to MB_{n-1}^\Sigma(\cl{X})\) discards the object \(x_i\) from the tuple indexing each summand, using maps provided by the monoidal structure of \(\Sigma\) and composition at the object \(x_i\):
\begin{align*}
\Sigma \cl{X}(x_{i-1},x_i) \otimes \Sigma \cl{X}(x_i, x_{i+1}) & \xto{\cong} \Sigma(\cl{X}(x_{i-1},x_i) \otimes \cl{X}(x_i, x_{i+1})) \\
& \xto{\Sigma(\circ)} \Sigma \cl{X}(x_{i-1}, x_{i+1}).
\end{align*}
The face map \(\delta_n^0\) discards the object \(x_0\) using the monoidal structure of \(\Sigma\) and the terminal map \(\cl{X}(x_0,x_1) \to I\) in \(\cf{V}\), which exists because \(\cf{V}\) is semicartesian. The map \(\delta_n^n\) discards \(x_n\) in the same way.
\item For each \(i\), the degeneracy map \(\sigma_n^i: MB_n^\Sigma(\cl{X}) \to MB_{n+1}^\Sigma(\cl{X})\) is induced by `inserting the identity' on \(x_i\), using the monoidal structure of \(\Sigma\):
\[\bb{I} \xto{\cong} \Sigma I \xto{\Sigma(\id_{x_i})} \Sigma X(x_i,x_i).\]
\end{itemize}
To show that these maps satisfy the simplicial identities makes use of the fact that the monoidal structure of \(\Sigma\) is symmetric and strong.
\end{defn}

From any simplicial object \(B\) in an abelian category one can form two chain complexes. The \emph{unnormalized} chain complex \(C(B)\) has \(C_n(B) = B_n\), with differential given by an alternating sum of face maps: \(\partial_n = \sum_{i=0}^n (-1)^i \delta_i\). The \emph{normalized} chain complex \(N(B)\), also known as the \emph{Moore complex}, is obtained by quotienting \(C(B)\) by a subcomplex `spanned by the degenerate simplices':
\[N(B) = C(B) / D(B)\]
where \(D_n(B) = \sum_i \sigma_i (C_{n-1}(B))\). The complex \(D(B)\) is always acyclic, so the unnormalized and normalized complexes are naturally quasi-isomorphic:
\[
H_\bullet(N(B)) \cong H_\bullet(C(B)).
\]
This is proved, for example, as Theorem 8.3.8 in Weibel \cite{WeibelIntroduction1994}.

\begin{defn}\label{def:MC_MH}
Let \(\cf{V}\), \(\ab{A}\) and \(\Sigma\) be as in \Cref{def:mag_nerve}. The \textbf{unnormalized magnitude complex} of a \(\cf{V}\)-category \(\cl{X}\) with respect to \(\Sigma\) is
\[\widetilde{MC}_\bullet^\Sigma(\cl{X}) = C(MB^\Sigma(\cl{X}))\]
and the \textbf{normalized magnitude complex} is
\[MC_\bullet^\Sigma(\cl{X}) = N(MB^\Sigma(\cl{X})).\]
The \textbf{magnitude homology} of \(\cl{X}\) with respect to \(\Sigma\) is
\[MH_\bullet^\Sigma(\cl{X}) = H_\bullet(MC^\Sigma(\cl{X})) \cong H_\bullet(\widetilde{MC}^\Sigma(\cl{X})).\]
\end{defn}

In this paper we will be interested in the magnitude homology of categories enriched in each of the three monoidal categories discussed in \Cref{sec:enriched_cats}, namely \((\cn{Set}, \times, \{*\})\), \((\cn{Truth}, \wedge, T)\) and \(([0, \infty], +, 0)\). The category \(\cn{Truth}\) can be embedded into \(\cn{Set}\) by sending the object \(T\) to \(\{*\}\) and \(F\) to the empty set, and this embedding is strong symmetric monoidal. On \(\cn{Set}\), we will always take as coefficients the free abelian group functor \(\Sigma: \cn{Set} \to \cn{Ab}\), and we obtain our coefficient functor on \(\cn{Truth}\) by restriction of \(\Sigma\). Thus, in practice, we treat preorders as ordinary categories.

\begin{eg}\label{eg:categ_ordinary}
The simplex category \(\Delta\) can be identified with the full subcategory of \(\cn{Cat}\) on nonempty finite linear orders; we write \([n]\) for the linear order on \(n+1\) objects. The \textbf{nerve} of a small category \(\cl{X}\) is then the simplicial set defined by
\[S(\cl{X})_n = \cn{Cat}([n], \cl{X}).\]
Concretely, the \(n\)-simplices of \(S(\cl{X})\) are \(n\)-tuples of composable morphisms in \(\cl{X}\). The \textbf{classifying space} of \(\cl{X}\), denoted \(\bb{B}\cl{X}\), is the geometric realization \(|S(\cl{X})|\).

The magnitude complex of \(\cl{X}\) with respect to the free abelian group functor \(\Sigma\) is precisely the complex of simplicial chains in the nerve of \(\cl{X}\). It follows that \(MH_\bullet^\Sigma(\cl{X})\) coincides with the singular homology of the classifying space \(\bb{B}\cl{X}\). In particular, the magnitude homology of a group (regarded as a one-object category) is its ordinary group homology, and the magnitude homology of a poset is the homology of its order complex.
\end{eg}

In simple cases the normalized magnitude complex has the following convenient description.

\begin{lem}\label{lem:normalized}
Suppose \(\cl{X}\) is a \(\cf{V}\)-category such that \(\cl{X}(x,x) \cong I\) for every \(x \in \cl{X}\). Then the normalized magnitude complex of \(\cl{X}\) is given in degree \(n\) by
\[MC_n^\Sigma(\cl{X}) = \bigoplus
\Sigma \cl{X} (x_0,x_1) \otimes \cdots \otimes \Sigma \cl{X}(x_{n-1}, x_n)\]
where the sum is over tuples \(x_0, \ldots x_n\) such that \(x_i \neq x_{i+1}\) for all \(i\).
\end{lem}

\begin{proof}
The condition on \(\cl{X}\) implies that the map \(\id_x: I \to \cl{X}(x,x)\) is an isomorphism for every \(x \in \cl{X}\); the same is of course true of the induced map \(\bb{I} \to \Sigma \cl{X}(x,x)\), and therefore of every degeneracy map. Taking the quotient of \(\widetilde{MC}_n^\Sigma(\cl{X})\) by the image of every degeneracy map thus sends to 0 every summand associated to a tuple \((x_0, \ldots, x_n)\) in which \(x_i = x_{i+1}\) for some \(i\).
\end{proof}

We record a miniature example for illustration and for later reference.

\begin{eg}\label{eg:S1}
Let \(\bb{S}^1\) denote the category
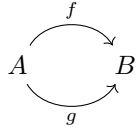
\bdiag
A \arrow[bend left=60]{r}{f} \arrow[swap, bend right=60]{r}{g} & B
\ediag
(Identity arrows not pictured.) The complex \(MC_\bullet^\Sigma(\bb{S}^1)\) is generated in degree 0 by the set of objects, \(\{A, B\}\). Applying \Cref{lem:normalized} and using the fact that \(\Sigma\bb{S}^1(B,A) = 0\), one sees that in degree 1 it is generated by \(\{f,g\}\), and in all higher degrees it vanishes. The boundary map \(\partial_1: \bb{Z} \cdot \{f,g\} \to \bb{Z} \cdot \{A, B\}\) is given on generators by \(\partial_1(f) = B - A = \partial_1(g)\); its image is generated by \(B-A\) and its kernel by \(f-g\). Thus, \(MH_0^\Sigma(\bb{S}^1) = \bb{Z} \cdot \{A,B\} / \langle B-A \rangle \cong \bb{Z}\) and \(MH_1^\Sigma(\bb{S}^1) = \langle f-g \rangle / 0 \cong \bb{Z}\). That is,
\[MH_\bullet^\Sigma(\bb{S}^1) = \begin{cases} \bb{Z} & *=0, 1 \\
 0 & \text{otherwise}. \end{cases}\]
\end{eg}

Magnitude homology is designed to categorify \emph{magnitude}: a numerical invariant of enriched categories with a remarkable web of connections to notions of `size' of relevance in different corners of mathematics. The study of magnitude precedes the construction of magnitude homology by almost a decade---originating with Leinster in \cite{LeinsterMagnitude2008}---and is especially well developed in the context of metric spaces. The magnitude of a metric space is best understood not just as a single number but as a real-valued function or power series \cite{LeinsterMagnitude2013}, which has been shown to carry information about the intrinsic volumes of a space (e.g.~\cite{BarceloCarbery2018, GimperleinGoffeng, MeckesMagnitude2019}), curvature and the Willmore energy \cite{GimperleinGoffengWillmore}, and Minkowski dimension \cite{MeckesMagnitude2015}. (These references are far from exhaustive.)

The desire for magnitude homology to recover, via its Euler characteristic, the magnitude of a (finite) metric space motivates a particular choice of coefficient functor on \(([0, \infty], +, 0)\), which we now describe.

\begin{defn}\label{def:AbR}
Let \(\cn{Ab}^{[0, \infty)}\) denote the category of abelian groups graded by the nonnegative reals. Given \(A \in \cn{Ab}^{[0,\infty)}\) we write \(A^\ell\) for the value of \(A\) at \(\ell \in [0, \infty)\), and refer to \(\ell\) as the \textbf{length grading} of \(A^\ell\).

The category \(\cn{Ab}^{[0,\infty)}\) inherits, grading-wise, many of the valuable properties of \(\cn{Ab}\)---in particular, since \(\cn{Ab}\) is abelian, so is \(\cn{Ab}^{[0,\infty)}\). It carries a closed symmetric monoidal structure given by the \textbf{convolution product} \(\ost\), for which
\[(A \ost B)^\ell = \bigoplus_{r + s = \ell} A^r \otimes_\bb{Z} B^s\]
for each \(\ell \in [0, \infty)\). The unit object has \(\bb{Z}\) in length grading 0, and 0 in all other gradings. There is a strong symmetric monoidal functor \(\Sigma^*: [0, \infty] \to \cn{Ab}^{[0, \infty)}\) given by \(\Sigma^*(\infty) = 0\) and, for \(t < \infty\),
\[\Sigma^\ell(t) = \begin{cases} \bb{Z} & \ell = t \\ 0 & \ell \neq t. \end{cases}\]
\end{defn}

The magnitude homology of a metric space is defined by Leinster and Shulman to be its magnitude homology with respect to \(\Sigma^*\) \cite{LeinsterMagnitude2017v3}. In the case of a graph, this agrees with the earlier definition given by Hepworth and Willerton \cite{HepworthWillerton2017}. \Cref{prop:MC_met} will describe the magnitude chain complex explicitly. To simplify notation we suppress \(\Sigma^*\), writing \(MC_k^\ell(X)\) and \(MH_k^\ell(X)\) for the magnitude complex---respectively, magnitude homology---of a metric space \(X\) in length grading \(\ell\) and homological degree \(k\).

\begin{defn}
A point \(z\) in a metric space \(X\) is said to be \textbf{between} points \(x\) and \(y\) if \(d(x,y) = d(x,z) + d(z,y)\). If, in addition, \(x \neq z \neq y\), we say \(z\) is \textbf{strictly between} \(x\) and \(y\). Points \(x,y \in X\) are \textbf{non-adjacent} if there exists a point \(z \in X\) lying strictly between them, and \textbf{adjacent} otherwise.
\end{defn}

\begin{prop}[\cite{LeinsterMagnitude2017v3}, Definition 3.3]\label{prop:MC_met}
The normalized magnitude chain complex of a metric space \(X\) is given for \(\ell \in [0, \infty)\) and \(n \in \bb{N}\) by
\[MC_n^\ell(X) = \bb{Z} \cdot \left\{(x_0, \ldots, x_n) \, \middle\vert \, x_i \neq x_{i+1} \text{ for all \(i\) and } \sum_{i=0}^{n-1} d(x_i, x_{i+1}) = \ell\right\}.\]
The differential is \(\partial = \sum_{i=1}^{n-1} (-1)^i \delta_i\) where \(\delta_i\) is specified on generators by
\[\delta_i(x_0, \ldots, x_n) = (x_0, \ldots, x_{i-1}, x_{i+1}, \ldots, x_n)\]
if \(x_i\) is between \(x_{i-1}\) and \(x_{i+1}\), and \(\delta_i(x_0, \ldots, x_n) = 0\) otherwise. \qed
\end{prop}

Since \(\partial_1 = 0\), we have \(MH_0^\ell(\cl{X}) = MC_0^\ell(\cl{X})\). In length grading 0 this group is freely generated by the set of points in \(X\), while in length gradings \(\ell > 0\) it vanishes. Magnitude homology in degree 1 is also straightforward to describe:

\begin{thm}[\cite{LeinsterMagnitude2017v3}, Theorem 4.3]\label{prop:MH1_met}
Let \(X\) be a metric space. Then \(MH_1^\ell(X)\) is the free abelian group on the set of ordered pairs \((x,y)\) of distinct points in \(X\) such that \(x\) and \(y\) are adjacent and \(d(x,y) = \ell\).
\end{thm}

The theorem implies, for example, that the first magnitude homology of a closed subset \(X\) of \(\bb{R}^n\) vanishes if and only if \(X\) is convex (\cite{LeinsterMagnitude2017v3}, Corollary 4.9); otherwise, magnitude homology records quantitative information about the `size' of its concavities. The higher magnitude homology groups similarly capture geometric, as opposed to topological, information. For instance, Asao has established that the second magnitude homology of a metric space detects the presence of closed geodesics, and that the curvature of a space strongly influences the gradings in which its magnitude homology of any degree can be non-vanishing (Theorem 5.3 and Corollary 4.6 in \cite{AsaoMH2019}).

To close this section, and to prepare the ground for the rest of the paper, we return to the general setting of semicartesianly enriched categories. The proof of \Cref{prop:MB_monoid} is straightforward, but lengthy, and just the outline is given here; details can be found in Appendix A of the author's thesis \cite{Roff2022}.

\begin{prop}\label{prop:MB_monoid}
Let \(\cf{V}\) be a semicartesian category, \(\ab{A}\) a symmetric monoidal closed abelian category, and \(\Sigma: \cf{V} \to \ab{A}\) a strong symmetric monoidal functor. The magnitude nerve defines a strong symmetric monoidal functor
\[MB^\Sigma: (\cf{V}\cn{Cat}, \otimes, \cl{I}) \to ([\Delta^\op, \ab{A}], \otimes_\pw, I_\pw).\]
\end{prop}

\begin{proof}[Sketch proof]
To see that there is an isomorphism of simplicial objects 
\(\Phi: {I}_\pw \xto{\sim} MB^\Sigma(\cl{I}),\)
note that the monoidal structure of \(\Sigma\) supplies an isomorphism \(\phi: {I}_\ab{A} \xto{\sim} \Sigma({I}_\cf{V})\), where \(I_\cf{V}\) denotes the unit in \(\cf{V}\). Hence, for each \(n \in \bb{N}\) there is an isomorphism
\[\Phi_n: {I}_\ab{A} \cong \underbrace{{I}_\ab{A} \otimes \cdots \otimes {I}_\ab{A}}_{n} \xto{\phi^{\otimes n}} \underbrace{ \Sigma({I}_\cf{V}) \otimes \cdots \otimes \Sigma({I}_\cf{V})}_{n} = MB_n(\cl{I})\]
and since the face maps and degeneracies of  \(MB^\Sigma(\cl{I})\) and \({I}_\pw\) are given by identities, these isomorphisms plainly define a map of simplicial objects.

To see that for each \(\cl{X}, \cl{Y} \in \ob(\cf{V}\cn{Cat})\) there is an isomorphism
\[\Psi_{\cl{X},\cl{Y}}: MB^\Sigma(\cl{X}) \otimes_\pw MB^\Sigma(\cl{Y}) \xto{\sim} MB^\Sigma(\cl{X} \otimes \cl{Y})\]
note that monoidal structure of \(\Sigma\) supplies a family of isomorphisms 
\[\psi_{V,W}: \Sigma(V) \otimes \Sigma(W) \xto{\sim} \Sigma(V \otimes W),\]
natural in \(V,W \in \ob(\cf{V})\). The assumption that the monoidal structure on \(\ab{A}\) is closed ensures that the tensor product distributes over arbitrary coproducts. Consequently there is, in each degree \(n\), an isomorphism
\[(\Psi_{\cl{X},\cl{Y}})_n: \left(MB^\Sigma(\cl{X}) \otimes_\pw MB^\Sigma(\cl{Y})\right)_n \to MB_n^\Sigma(\cl{X} \otimes \cl{Y})\]
obtained as the following composite, in which the two unnamed isomorphisms are due, respectively, to the distributivity of the tensor product and to the symmetry of the monoidal structure in \(\ab{A}\):
\begin{equation*}
\adjustbox{scale=1,center}{%
\begin{tikzcd}[row sep=1.5em]
\left(MB^\Sigma(\cl{X}) \otimes_\pw MB^\Sigma(\cl{Y})\right)_n
	\arrow[equals]{d} \\ [-5]
MB_n^\Sigma(\cl{X}) \otimes_{} MB_n^\Sigma(\cl{Y})
	\arrow[equals]{d} \\ [-5]
\displaystyle{\left( \bigoplus_{\substack{x_0,\ldots,x_n \\ \in \cl{X}}} \Sigma\cl{X}(x_0,x_1) \otimes \cdots \otimes \Sigma\cl{X}(x_{n-1},x_n) \right) \otimes \left( \bigoplus_{\substack{y_0,\ldots,y_n \\ \in \cl{Y}}} \Sigma\cl{Y}(y_0,y_1) \otimes \cdots \otimes \Sigma\cl{Y}(y_{n-1},y_n) \right)}
	\arrow{d}{\cong} \\
\displaystyle{\bigoplus_{\substack{x_0,\ldots,x_n \in \cl{X} \\ y_0,\ldots, y_n \in \cl{Y}}}\left(  \Sigma\cl{X}(x_0,x_1) \otimes \cdots \otimes \Sigma\cl{X}(x_{n-1},x_n) \otimes \Sigma\cl{Y}(y_0,y_1) \otimes \cdots \otimes \Sigma\cl{Y}(y_{n-1},y_n)\right)}
	\arrow{d}{\cong} \\
\displaystyle{\bigoplus_{\substack{x_0,\ldots,x_n \in \cl{X} \\ y_0,\ldots, y_n \in \cl{Y}}} \left(\Sigma\cl{X}(x_0,x_1) \otimes \Sigma\cl{Y}(y_0,y_1)\right) \otimes \cdots \otimes \left(\Sigma\cl{X}(x_{n-1},x_n) \otimes \Sigma\cl{Y}(y_{n-1},y_n)\right)} 
	\arrow{d}{\bigoplus \psi_{\cl{X}(x_0,x_1),\cl{Y}(y_0,y_1)} \otimes \cdots \otimes \psi_{\cl{X}(x_{n-1},x_n),\cl{Y}(y_{n-1},y_n)}} \\
\displaystyle{\bigoplus_{\substack{x_0,\ldots,x_n \in \cl{X} \\ y_0,\ldots, y_n \in \cl{Y}}} \Sigma\left(\cl{X}(x_0,x_1) \otimes \cl{Y}(y_0,y_1)\right) \otimes \cdots \otimes \Sigma\left(\cl{X}(x_{n-1},x_n) \otimes \cl{Y}(y_{n-1},y_n)\right)}
	\arrow[equals]{d}{} \\ [-5]
\displaystyle{\bigoplus_{\substack{(x_0,y_0),\ldots,(x_n,y_n) \\ \in \cl{X} \otimes_{\cf{V}} \cl{Y}}}\Sigma(\cl{X} \otimes \cl{Y})((x_0,y_0),(x_1,y_1)) \otimes \cdots \otimes \Sigma(\cl{X} \otimes \cl{Y}) ((x_{n-1},y_{n-1}),(x_n,y_n))} 
	\arrow[equals]{d} \\ [-5]
MB_n^\Sigma(\cl{X} \otimes \cl{Y}).
\end{tikzcd}
}
\end{equation*}

To verify that the family \(\Psi_{\cl{X},\cl{Y}} = \left((\Psi_{\cl{X},\cl{Y}})_n \st n \in \bb{N}\right)\) forms a map of simplicial objects, one must check its naturality with respect to face maps and degeneracies; it remains then to verify that the family \(\Psi = \left(\Psi_{\cl{X}, \cl{Y}} \st \cl{X}, \cl{Y} \in \cf{V}\cn{Cat}\right)\) is natural with respect to \(\cf{V}\)-functors \(\cl{X} \to \cl{X}'\) and \(\cl{Y} \to \cl{Y}'\) and satisfies the relevant coherence axioms. Each of these verifications is straightforward, with the coherence of the monoidal structure on \(\ab{A}\) reducing the check to an application of the monoidal structure of \(\Sigma\).
\end{proof}


\section{The K\"unneth theorem}\label{sec:Kunneth}

As a first application of the monoidal structure of the magnitude nerve (\Cref{prop:MB_monoid}), in this section we prove a general K\"unneth theorem for the magnitude homology of semicartesianly enriched categories.

The classical K\"unneth theorem relates the singular homology of a product of two topological spaces \(X\) and \(Y\) to the tensor product of their homologies. It is proved as the composite of two results: the topological Eilenberg--Zilber theorem, which says that the singular chain complex \(C_\bullet(X \times Y)\) is quasi-isomorphic to the tensor product \(C_\bullet(X) \otimes C_\bullet(Y)\), and an algebraic K\"unneth theorem, which relates the homology of a tensor product of chain complexes to the tensor product of their homologies. (E.g.~Hatcher \cite{HatcherAlgebraic2001}, \S3.B.)

Our proof of the K\"unneth theorem for magnitude homology has the same structure. Using the monoidal structure of the magnitude nerve, we first prove an Eilenberg--Zilber-type theorem for the magnitude chain complex, and from there we can derive the main result via the algebraic K\"unneth theorem. Underlying our Eilenberg--Zilber theorem is the following fundamental fact about double complexes arising from bisimplicial objects, which is proved as Theorem 8.5.1 in Weibel \cite{WeibelIntroduction1994}.

\begin{thm}[Bisimplicial Eilenberg--Zilber theorem]\label{thm:EZ_general}
Let \(B\) be a bisimplicial object in an abelian category. Then there 
is a natural isomorphism 
\[
H_\bullet N(\diag(B)) \cong H_\bullet\Tot(CB).
\]
Moreover, there exist natural chain maps
\[
C(\diag(B)) \rightleftarrows \Tot(CB)
\]
which realize this isomorphism on homology. \qed
\end{thm}

Eventually we will want to use the full strength of \Cref{thm:EZ_general}---indeed, it will be central to many of the arguments and computations in later sections. For now, though, we are concerned with a special case which is perhaps more immediately suggestive of its connection to the topological Eilenberg--Zilber theorem. Here, \(\Ch(\ab{A})\) denotes the category of chain complexes in \(\ab{A}\), and \(\otimes_{\Ch(\ab{A})}\) denotes the standard tensor product of chain complexes.

\begin{cor}\label{cor:EZ_pw}
Let \(\ab{A}\) be a symmetric monoidal abelian category. For every \(A, A' \in [\Delta^\op, \ab{A}]\) there is an isomorphism
\[
H_\bullet C(A \otimes_\pw A') \cong H_\bullet (C(A) \otimes_{\Ch(\ab{A})} C(A'))
\]
natural in \(A\) and \(A'\). Moreover, there exist natural chain maps
\[C(A \otimes_\pw A') \rightleftarrows C(A) \otimes_{\Ch(\ab{A})} C(A')\]
realizing this isomorphism on homology.
\end{cor}

\begin{proof}
Given any two simplicial objects \(A\) and \(A'\) in \(\ab{A}\), the composite
\[\Delta^\op \times \Delta^\op \xto{A \times A'} \ab{A} \times \ab{A} \xto{\otimes} \ab{A}.\]
is a bisimplicial object. Denote it by \(A \boxtimes A'\); so \((A \boxtimes A')_{nm} = A_n \otimes A'_m\), and the vertical face and degeneracy operators are those of \(A\), while the horizontal operators are those of \(A'\). Taking the diagonal of this bisimplicial object yields the pointwise tensor product of the functors \(A\) and \(A'\): 
\[A \otimes_\pw A' = \diag(A \boxtimes A').\]
Meanwhile, the total complex of \(C(A \boxtimes A')\) is given in degree \(n\) by
\[\Tot_nC(A \boxtimes A') = \bigoplus_{j+k=n} A_j \otimes A'_k\]
and on comparing differentials one sees there is a chain complex isomorphism
\[\Tot\: C(A \boxtimes A') \cong C(A) \otimes_{\Ch(\ab{A})} C(A').\]
The result now follows from \Cref{thm:EZ_general}, taking \(B = A \boxtimes A'\) and making use of the fact that the normalized and unnormalized chain complexes of any simplicial object in \(\ab{A}\) are naturally quasi-isomorphic (Theorem 8.3.8 in \cite{WeibelIntroduction1994}).
\end{proof}

Combining \Cref{cor:EZ_pw} with our \Cref{prop:MB_monoid} delivers the Eilenberg--Zilber theorem for the magnitude chain complex, as follows.

\begin{thm}[Eilenberg--Zilber theorem for magnitude chains]\label{thm:EZ_magnitude}
Let \(\cf{V}\) be a semicartesian category, let \(\ab{A}\) be a symmetric monoidal closed abelian category and let \(\Sigma: \cf{V} \to \ab{A}\) be a strong symmetric monoidal functor. Given any two \(\cf{V}\)-categories \(\cl{X}\) and \(\cl{Y}\) there exist chain maps
\[\widetilde{MC}^\Sigma(\cl{X} \otimes \cl{Y}) \rightleftarrows \widetilde{MC}^\Sigma(\cl{X}) \otimes_{\Ch(\ab{A})} \widetilde{MC}^\Sigma(\cl{Y}) \]
natural in \(\cl{X}\) and \(\cl{Y}\), inducing an isomorphism on homology:
\begin{equation}\label{EZ_magnitude}
MH_\bullet^\Sigma(\cl{X} \otimes \cl{Y}) \cong H_\bullet\left(\widetilde{MC}^\Sigma(\cl{X}) \otimes_{\Ch(\ab{A})} \widetilde{MC}^\Sigma(\cl{Y})\right).
\end{equation}
\end{thm}

\begin{proof}
The strong monoidal structure of the magnitude nerve (\Cref{prop:MB_monoid}) supplies an isomorphism
\[MB^\Sigma(\cl{X} \otimes \cl{Y}) \cong MB^\Sigma(\cl{X}) \otimes_\pw MB^\Sigma(\cl{Y}),\]
natural in \(\cl{X}\) and \(\cl{Y}\). Taking the unnormalized chain complex on both sides and applying \Cref{cor:EZ_pw} on the right yields natural chain maps
\[\widetilde{MC}^\Sigma(\cl{X} \otimes \cl{Y}) \rightleftarrows C(MB^\Sigma(\cl{X}) \otimes_\pw MB^\Sigma(\cl{Y})) \rightleftarrows \widetilde{MC}^\Sigma(\cl{X}) \otimes_{\Ch(\ab{A})} \widetilde{MC}^\Sigma(\cl{Y}) \]
inducing the claimed isomorphism on homology.
\end{proof}

To relate the right hand side of the isomorphism in (\ref{EZ_magnitude}) to a tensor product of homologies requires an algebraic K\"unneth theorem. We will employ it in the following form---proved, for instance, as Theorem 2.4 in May \cite{MayTorExt}.

\begin{thm}\label{thm:Kunneth_A}
Let \(R\) be a principal ideal domain and \(\ab{A}=\cn{Mod}_R\) its category of modules. Let \(C \in \Ch(\ab{A})\) be a chain complex of flat modules. Then given any other complex \(D \in \Ch(\ab{A})\), for each \(n \in \bb{N}\) there is a short exact sequence
\begin{align*}
0 \to \bigoplus_k  H_k(C) \otimes_R H_{n-k}(D) &\to H_n\left(C \otimes_{\Ch(\ab{A})} D\right) \\
&\to \bigoplus_k \Tor\left(H_k(C), H_{n-k-1}(D)\right) \to 0
\end{align*}
natural in \(C\) and \(D\). The sequence splits, but the splitting is not natural. \qed
\end{thm}

Combining \Cref{thm:Kunneth_A} with \Cref{thm:EZ_magnitude} establishes the K\"unneth theorem for magnitude homology. As \Cref{thm:Kunneth_A} is stated for categories of modules over a principal ideal domain, that is the context in which our theorem holds.

\begin{thm}[K\"unneth theorem for magnitude homology]\label{thm:MH_Kunneth}
Let \(\cf{V}\) be a semicartesian category and \(R\) a principal ideal domain, and let \(\Sigma: \cf{V} \to \cn{Mod}_R\) be a strong symmetric monoidal functor. Suppose \(\cl{X}\) is a \(\cf{V}\)-category with the property that, for every \(n \in \bb{N}\), the \(R\)-module \({MC}_n^\Sigma(\cl{X})\) is flat. Then given any \(\cf{V}\)-category \(\cl{Y}\), for each \(n \in \bb{N}\) there is a short exact sequence
\begin{align*}
0 \to \bigoplus_k MH_k^\Sigma(\cl{X}) \otimes_{R} MH_{n-k}^\Sigma(\cl{Y}) &\to MH_n^\Sigma(\cl{X} \otimes \cl{Y}) \\
&\to \bigoplus_{k} \Tor\left(MH_k^\Sigma(\cl{X}), MH_{n-k-1}^\Sigma(\cl{Y})\right) \to 0
\end{align*}
natural in \(\cl{X}\) and \(\cl{Y}\). The sequence splits, but the splitting is not natural.
\end{thm}

\begin{proof}
\Cref{thm:Kunneth_A} says there is a natural short exact sequence
\begin{align*}
0 \to \bigoplus_k MH_k^\Sigma(\cl{X}) \otimes_R MH_{n-k}^\Sigma(\cl{Y}) & \to H_n\left(\widetilde{MC}^\Sigma(\cl{X}) \otimes_{\Ch(\cn{Mod}_R)} \widetilde{MC}^\Sigma(\cl{Y})\right) \\
&\to \bigoplus_{k} \Tor\left(MH_k^\Sigma(\cl{X}), MH_{n-k-1}^\Sigma(\cl{Y})\right) \to 0
\end{align*}
which splits. \Cref{thm:EZ_magnitude} tells us the middle term is naturally isomorphic to \(MH_n^\Sigma(\cl{X} \otimes \cl{Y})\), and this gives the result.
\end{proof}

Because the usual choice of coefficient functor \(\Sigma^*\) on \([0, \infty]\) is not valued in a category of modules over a ring but in a category of graded modules, \Cref{thm:MH_Kunneth} does not immediately specialize to metric spaces. However, the Eilenberg--Zilber theorem for magnitude chains still applies in this setting, providing a K\"unneth formula for generalized metric spaces, as follows. In the case that \(X\) and \(Y\) are classical metric spaces (in particular, symmetric), \Cref{thm:MH_Kunneth_metric} recovers Proposition 4.3 of \cite{BottinelliKaiser2021}, itself a generalization of Theorem 5.1 in \cite{HepworthWillerton2017} (the K\"unneth formula for undirected graphs).

\begin{thm}\label{thm:MH_Kunneth_metric}
Let \(X\) be a generalized metric space. Given any other generalized metric space \(Y\), there is  in each length grading \(0 \leq \ell < \infty\) a short exact sequence
\begin{align*}
0 \to \bigoplus_{\substack{r+s = \ell \\ j+k = n}} MH_j^r ({X}) \otimes_\bb{Z} MH_k^s({Y}) & \to MH_{n}^{\ell}(X \otimes Y) \\
&\to \bigoplus_{\substack{r+s = \ell \\ j+k = n}} \Tor\left(MH_{j}^{r}({X}), MH_{k-1}^{s}({Y})\right) \to 0
\end{align*}
natural in \(X\) and \(Y\). The sequence splits, but the splitting is not natural.
\end{thm}

\begin{proof}
\Cref{thm:EZ_magnitude} tells us that for spaces \(X\) and \(Y\) there is a natural isomorphism of bigraded abelian groups
\begin{equation}\label{EZ_metric}
MH_{\bullet}^*({X} \otimes {Y}) \cong H_{\bullet}\left(\widetilde{MC}({X}) \ost_{\Ch} \widetilde{MC}({Y})\right)
\end{equation}
where \(\ost_{\Ch}\) denotes the tensor product of chain complexes in \(\cn{Ab}^{[0, \infty)}\). The magnitude complex of any generalized metric space is free, and hence flat, in every length grading. We can therefore apply the algebraic K\"unneth theorem grading-wise on the right of (\ref{EZ_metric}) to obtain the result.
\end{proof}


\section{Iterated magnitude homology}\label{sec:iterated_MH}

We turn now to the main theme of the paper, which is to construct a modified `magnitude homology' theory applicable to categories with a second- or higher-order enrichment---categories enriched in some category of \(\cf{V}\)-categories---where a coefficient functor on \(\cf{V}\) has already been chosen. Our template for this construction will be the  \emph{classifying space} \(\bb{B}\cl{X}\) of a bicategory \(\cl{X}\).

In modern texts, the {classifying space} of a bicategory is usually defined to be the geometric realization of its \emph{geometric nerve}: the simplicial set \(S(\cl{X})\) specified by
\[S(\cl{X})_n = \cn{Bicat}([n], \cl{X})\]
where \(\cn{Bicat}\) is the 1-category of bicategories whose morphisms are lax 2-functors that preserve identities strictly. This definition is found implicitly in Street \cite{StreetAlgebra1987} in 1987 and given its succinct description by Duskin \cite{DuskinNerve2002} in 2002.

This construction has the virtue that it evidently extends the ordinary classifying space functor along the inclusion of small categories into bicategories. When it comes to strict 2-categories, though, there is at least one other way one might think about constructing something like a classifying space. Segal \cite{SegalClassifying1968} observed in 1968 that the nerve \(N\cl{C}\) of any topological category \(\cl{C}\) (i.e.~any category internal to \(\cn{Top}\)) naturally carries the structure of a simplicial space, and he defined the classifying space of a topological category to be the geometric realization of that simplicial space, \(|N\cl{C}|\). Meanwhile, if \(\cl{X}\) is a strict 2-category, then applying the {ordinary} classifying space functor {locally}---that is, to each hom-category in \(\cl{X}\)---transforms \(\cl{X}\) into a topological category \(\bb{B}_*\cl{X}\) with discrete space of objects. Segal's construction then gives an alternative `classifying space' functor
\[|N\bb{B}_*(-)|: 2\cn{Cat} \to \cn{Top}\]
where \(2\cn{Cat}\) is the category of strict 2-categories.

Although the Segal construction and the Duskin--Street construction look rather different, where both are applicable they give equivalent results: Bullejos and Cegarra prove in Theorem 1 of \cite{BullejosCegarra2003} that for any strict 2-category \(\cl{X}\) there is a natural homotopy equivalence \(\bb{B}\cl{X} \simeq |N\bb{B}_*\cl{X}|\).

We are going to follow Segal's lead, using the magnitude nerve in place of the ordinary classifying space functor, to define a homology theory for categories enriched in \(\cf{V}\cn{Cat}\), where \(\cf{V}\) is any semicartesian category. First we need to verify that \(\cf{V}\cn{Cat}\) is semicartesian, provided \(\cf{V}\) is.

\begin{lem}\label{lem:VCat_semi}
Let \(\cf{V}\) be a semicartesian category. Then the monoidal structure on \(\cf{V}\cn{Cat}\) induced by that of \(\cf{V}\) is semicartesian.
\end{lem}

\begin{proof}
The unit for the monoidal product on \(\cf{V}\cn{Cat}\) is the one-object \(\cf{V}\)-category \(\cl{I}\) whose sole hom-object is the unit \(I\) in \(\cf{V}\). Given any \(\cf{V}\)-category \(\cl{C}\) we can define a \(\cf{V}\)-functor \(T: \cl{C} \to \cl{I}\) by sending each object in \(\cl{C}\) to the unique object in \(\cl{I}\); as \(I\) is terminal, there is for each \(a,b \in \ob (\cl{C})\) a unique map \(\cl{C}(a,b) \to \cl{I}(Ta, Tb) = I\). This determines \(T\) uniquely, so \(\cl{I}\) is terminal in \(\cf{V}\cn{Cat}\).
\end{proof}

Now let \(\ab{A}\) be a symmetric monoidal closed abelian category and \(\Sigma: \cf{V} \to \ab{A}\) a strong symmetric monoidal functor. Let \(\cl{X}\) be a category enriched in \(\cf{V}\cn{Cat}\). As the magnitude nerve functor \(MB^\Sigma: \cf{V}\cn{Cat} \to [\Delta^\op, \ab{A}]\) is strong symmetric monoidal (\Cref{prop:MB_monoid}), we can consider the magnitude nerve of \(\cl{X}\) with respect to \(MB^\Sigma\). According to \Cref{def:mag_nerve}, this is a simplicial object in \([\Delta^\op, \ab{A}]\); in other words, \(MB^{MB^{\Sigma}}(\cl{X})\) is a bisimplicial object in \(\ab{A}\). 

\begin{defn}\label{def:2MH}
The \textbf{double magnitude nerve} of a \(\cf{V}\cn{Cat}\)-category \(\cl{X}\) is 
\[MB^{MB^{\Sigma}}(\cl{X}) \in [\Delta^\op \times \Delta^\op, \ab{A}].\]
The \textbf{iterated magnitude nerve} of \(\cl{X}\) is the simplicial object in \(\ab{A}\)
\[MB^{2\Sigma}(\cl{X}) = \diag \: MB^{MB^{\Sigma}}(\cl{X}),\]
the \textbf{iterated magnitude complex} of \(\cl{X}\) is the chain complex
\[MC^{2\Sigma}(\cl{X}) = C(MB^{2\Sigma}(\cl{X})),\]
and the \textbf{iterated magnitude homology} of \(\cl{X}\) is its homology:
\[MH^{2\Sigma}_\bullet(\cl{X}) = H_\bullet (MC^{2\Sigma}(\cl{X})).\]
\end{defn}

The bisimplicial Eilenberg--Zilber theorem (\Cref{thm:EZ_general}) implies that iterated magnitude homology can equally well be realized as the total homology of the chains in the double magnitude nerve: there is a natural isomorphism
\begin{equation}\label{MH2_EZ}
MH_\bullet^{2\Sigma}(-) \cong H_\bullet \Tot\: C\left(MB^{MB^{\Sigma}}(-)\right).
\end{equation}
It will often be this description that we use in computations. However, it is the definition in terms of the iterated magnitude nerve that lets us compare the iterated magnitude homology of a strict 2-category with the homology of its classifying space.

\begin{thm}\label{thm:MH_Duskin}
Let \(\Sigma: \cn{Set} \to \cn{Ab}\) be the free abelian group functor. Then for every strict 2-category \(\cl{X}\) we have \(MH_\bullet^{2\Sigma}(\cl{X}) \cong H_\bullet(\bb{B}\cl{X})\), naturally in \(\cl{X}\).
\end{thm}

\begin{proof}
By design, the double magnitude nerve of \(\cl{X}\) with respect to \(\Sigma\) is the bisimplicial abelian group freely generated by the `double nerve' \(NN\cl{X}\) that Bullejos and Cegarra describe in \cite{BullejosCegarra2003} (pages 9--10). It is shown in the proof of their Theorem 1 that the geometric realization of \(\diag(NN\cl{X})\) is naturally homotopy equivalent to \(\bb{B}\cl{X}\). Since \(MC^{2\Sigma}_\bullet(\cl{X})\) is the Moore complex of \(MB^{2\Sigma}_\bullet(\cl{X}) = \bb{Z} \cdot \diag(NN\cl{X})\), this gives the result.
\end{proof}


\section{Enriched groups}\label{sec:enriched_groups}

The remainder of the paper is dedicated to investigating the information carried by iterated magnitude homology when it is interpreted in two rather different contexts. The first of these is drawn from group theory, where second-order enrichment occurs quite naturally; in this section we outline three classes of examples. The main definition is as follows.
\begin{defn}
A \textbf{\(\cf{V}\cn{Cat}\)-group} is a one-object \(\cf{V}\cn{Cat}\)-category---equivalently, a monoid object in \((\cf{V}\cn{Cat}, \otimes, \cl{I})\)---whose underlying ordinary category is a group.
\end{defn}

In other words, a \(\cf{V}\cn{Cat}\)-group is a group whose set of elements also carries the structure of a \(\cf{V}\)-category, and in which multiplication is a \(\cf{V}\)-functor. Inversion is \emph{not} required to be a \(\cf{V}\)-functor.

If the monoidal structure on \(\cf{V}\cn{Cat}\) happens to be cartesian, one can also talk about its \emph{internal group objects}, and these can be characterized as those \(\cf{V}\cn{Cat}\)-groups in which inversion is a \(\cf{V}\)-functor. If \(\cf{V}\cn{Cat}\) is not cartesian---for instance, if \(\cf{V}\cn{Cat}\) is \(\cn{GMet}\)---the absence of diagonal maps means that internal groups are not defined (see \Cref{rmk:internal_met}). Internal groups are a more widely familiar concept, well-motivated by many examples in geometry and topology, but the notion of \(\cf{V}\cn{Cat}\)-enrichment is also a natural one to consider, as our examples will illustrate. 

The study of \(\cf{V}\cn{Cat}\)-groups has recent precedent. Clementino, Martins-Ferreira and Montoli's 2019 paper \cite{ClementinoEtal2019} studies the category of preordered groups, which are groups with a second-order enrichment in \((\cn{Truth}, \wedge)\). Clementino and Montoli \cite{ClementinoMontoli2021} have since extended the investigation to groups with a second-order enrichment in any quantale, which encompasses the case of groups equipped with a translation-invariant metric (see \Cref{eg:Metgroups}, below) but not, for instance, any strict 2-group whose category of elements is not a preorder (see \Cref{eg:2groups}).
What we call \(\cf{V}\cn{Cat}\)-groups, Clementino and Montoli call `\(\cf{V}\)-groups'---we have chosen the slightly more unwieldy term to emphasise the second-order enrichment and specifically to avoid referring to strict 2-groups as `\(\cn{Set}\)-groups'. The following statement is proved as Proposition 3.1 in \cite{ClementinoMontoli2021}.

\begin{prop}\label{prop:enrich_biinv}
Let \(\cf{V}\) be a quantale, and let \(G\) be a group whose set of elements carries the structure of a \(\cf{V}\)-category. That structure makes \(G\) a \(\cf{V}\cn{Cat}\)-group if and only if it is \emph{translation invariant}, meaning that for every \(g \in G\) the maps \(g\cdot-: G \to G\) and \(-\cdot g: G \to G\) are both \(\cf{V}\)-functors. \qed
\end{prop}

In the next section we will analyse the iterated magnitude homology in low degrees of a \(\cf{V}\cn{Cat}\)-group where \(\cf{V} = \cn{Set}\) (so \(\cf{V}\cn{Cat} = \cn{Cat}\)), \(\cf{V} = \cn{Truth}\) (so \(\cf{V}\cn{Cat} = \cn{PreOrd}\)), or \(\cf{V} = [0, \infty]\) (so \(\cf{V}\cn{Cat} = \cn{GMet}\)). To make the analysis as concrete as possible, we will first describe three types of naturally occurring example: strict 2-groups, partially ordered groups, and any group equipped with a norm that is constant on conjugacy classes.

\begin{eg}\label{eg:2groups}
The class of \(\cn{Cat}\)-groups encompasses all groups internal to \(\cn{Cat}\). These are usually called  \emph{strict 2-groups} \cite{BaezLauda2004} or \emph{strict \(gr\)-categories} \cite{Sinh1978}.  Explicitly, a strict 2-group is a category \(\cl{G}\) equipped with functors \(m: \cl{G} \times \cl{G} \to \cl{G}\) (performing multiplication), \(e: \cl{1} \to \cl{G}\) (selecting the identity) and \(i: \cl{G} \to \cl{G}\) (taking inverses) which satisfy axioms imitating those for an ordinary group. In particular, the action of \(m\), \(e\) and \(i\) on objects makes the set \(\ob(\cl{G})\) into a group---and these functors make the set \(\arr(\cl{G})\) of arrows into a group as well. Given arrows \(p: g_1 \to g_2\) and \(q: h_1 \to h_2\), the function
\[
(\cl{G} \times \cl{G}) ( (g_1,h_1), (g_2,h_2)) = \cl{G}(g_1,g_2) \times \cl{G}(h_1,h_2) \xto{m_{(g_1,h_1),(g_2,h_2)}} \cl{G}(g_1h_1, g_2h_2)
\]
provided by the functoriality of \(m\) tells us how to multiply \(p\) and \(q\). The identity in \(\arr(\cl{G})\) is the identity arrow on the identity object \(e\); the inverse of an arrow \(g_1 \to g_2\) is an arrow \(g_1^{-1} \to g_2^{-1}\).

So a strict 2-group really does comprise two groups---a group of objects and a group of arrows---which are intertwined in a specific way. The rules for this intertwining are captured by the data of a \emph{crossed module}. A crossed module is a group homomorphism \(\phi: H \to G\) along with an action of \(G\) on \(H\) by automorphisms, such that \(\phi\) is equivariant with respect to conjugation in \(G\) and satisfies an equation known as the \emph{Pfeiffer identity}. These conditions ensure, in particular, that \(\im(\phi)\) is a normal subgroup of \(G\). Indeed, the simplest example of a crossed module is given by the inclusion of a normal subgroup \(N \hookrightarrow G\). In the corresponding strict 2-group \(\cl{G}_N\), the objects are the elements of \(G\) and an arrow \(g \to h\) is a pair \((k,g)\) such that \(k \in N\) and \(kg = h\). Composition of arrows is given by multiplication in \(N\), and the functor \(m: \cl{G}_N \times \cl{G}_N \to \cl{G}_N\) is defined on objects by multiplication in \(G\) and on arrows by multiplication in the semidirect product \(N \rtimes G\).

We can travel in the other direction, too, to construct a crossed module from any strict 2-group \(\cl{G}\). This makes use of the source and target maps \(s,t: \arr(\cl{G}) \rightrightarrows \ob(\cl{G})\), both of which are group homomorphisms. Let \(G = \ob(\cl{G})\) and \(H = \ker(s)\); that is, \(H = \{ p: e \to g \mid g \in G\}\). The target map restricts to a homomorphism \(t: H \to G\) and one can define an action of \(G\) on \(H\) making \(t\) into a crossed module. This construction extends to an equivalence between the category of strict 2-groups and a certain category of crossed modules; for a textbook account, see \S12.8 of Mac Lane \cite{MacLane1997}.

That equivalence is of interest---and will be useful in the next section---because crossed modules have played an important role historically in homotopy theory. They were introduced by Whitehead in the 1940s as an algebraic model for \emph{homotopy 2-types}: spaces whose homotopy groups \(\pi_n\) vanish for \(n > 2\) \cite{Whitehead1949}. From any crossed module \(H \xto{\phi} G\) one can build a space \(\bb{B}(H \xto{\phi} G)\) with the property that \(\pi_1(\bb{B}(H \xto{\phi} G)) \cong G/\im(\phi)\), \(\pi_2(\bb{B}(H \xto{\phi} G)) \cong \ker(\phi)\) and \(\pi_n(\bb{B}(H \xto{\phi} G)) = 0\) for all \(n > 2\). This, in turn, extends to an equivalence between a certain homotopy category of crossed modules and the homotopy category of pointed, connected homotopy 2-types \cite{MacLaneWhitehead1950}.

For example, given a normal subgroup \(N \hookrightarrow G\), the space \(\bb{B}(N \hookrightarrow G)\) has fundamental group \(G/N\). Given any strict 2-group \(\cl{G}\), the fundamental group of the space \(\bb{B}(H \xto{t} G)\) is \(G/E\), where \(E = \{g \in G \mid \exists \, e \xto{p} g\}\).
\end{eg}

\begin{eg}\label{eg:POgroups}
A \emph{partially ordered group} is a group \(G\) equipped with a partial order on its elements that is translation invariant in the sense that for all \(g \leq h\) and all \(k \in G\) we have
\[gk \leq hk \text{ and } kg \leq kh.\]
The study of partially ordered groups has a rich history going back to work of Birkhoff, Clifford and others in the 1940s (e.g.~\cite{Birkhoff1942,Clifford1940,EverettUlam1945}), and further still in the case of ordered fields.

Since monotone maps are \(\cn{Truth}\)-functors, \Cref{prop:enrich_biinv} tells us that every partially ordered group is a \(\cn{PreOrd}\)-group. In particular, every partially ordered group is a \(\cn{Cat}\)-group---but the interesting partially ordered groups are \emph{not} strict 2-groups. Indeed, a partially ordered group is a strict 2-group if and only if inversion is a monotone map, in which case for all \(g,h \in G\) such that \(g \leq h\) we have \(g^{-1} \leq h^{-1}\) and thus, by translation invariance, \(h = h g^{-1} g \leq h h^{-1} g = g\). In other words, only the trivial partial order on \(G\) makes \(G\) into a strict 2-group.

In \cite{ClementinoEtal2019}, Clementino \emph{et al} study \(\cn{PreOrd}\)-groups in generality. Central to their analysis is the fact that to give a group \(G\) a \(\cn{PreOrd}\)-enrichment is equivalent to specifying a submonoid of \(G\) which is closed under conjugation (\cite{ClementinoEtal2019}, Proposition 2.1). This submonoid is the \textbf{positive cone} \(P_G = \{g \in G \mid e \leq g \}\). A preorder makes \(G\) into a strict 2-group if and only if \(P_G\) is a normal subgroup, which holds if and only if the preorder is symmetric and therefore a congruence on \(G\).
\end{eg}

\begin{eg}\label{eg:Metgroups}
When group theorists speak of a `left-invariant' or `right-invariant' metric on a group, they mean a metric with respect to which left-translation (respectively, right-translation) by any given element is an isometry. A `bi-invariant' metric is one which is both left- and right-invariant in this sense. This appears stronger than the notion of translation invariance in \Cref{prop:enrich_biinv}, which only requires that translation be an enriched functor---in the case of metric spaces, this means a short map. In fact, the two notions are equivalent: if right-translation by every element is a {short map}, then for every \(g\), \(h\) and \(k\) in \(G\) we have
\[d(g,h) = d(gkk^{-1}, hkk^{-1}) \leq d(gk,hk) \leq d(g,h)\]
so right-translation by \(k\) is an isometry, and in the same way so is left-translation. By \Cref{prop:enrich_biinv}, then, a bi-invariant metric on a group is the same thing as an enrichment in the category \(\cn{Met}\) of metric spaces and short maps.

In \Cref{sec:MH_groups} we will see that the iterated magnitude homology of a \(\cn{Met}\)-group is easiest to describe in terms of the associated group norm. A \textbf{group norm} on a group \(G\) with identity element \(e\) is a function \(|-|:G \to \bb{R}\) satisfying
\begin{enumerate}
\item \(|e| = 0\) and \(|g| > 0\) for all \(g \in G \backslash \{e\}\); \label{def:mon_norm_1}
\item \(|gh| \leq |g| + |h|\) for all \(g,h \in G\). \label{def:mon_norm_2}
\end{enumerate}
If in addition \(|ghg^{-1}| = |h|\) for all \(g,h \in G\)---equivalently, \(|gh| = |hg|\) for all \(g\) and \(h\)---one says that \(|-|\) is \textbf{conjugation invariant}.

Any group norm on \(G\) induces a metric by setting \(d(g,h) = |h^{-1}g|\): separatedness follows from condition \ref{def:mon_norm_1} in the definition, and the triangle inequality from condition \ref{def:mon_norm_2}, just as in the case of a vector space norm. It is well known and elementary that this metric is guaranteed always to be left-invariant, and will be right-invariant too if the norm is conjugation invariant. Conversely, any left-invariant metric \(d\) on \(G\) induces a norm on \(G\) by setting \(|g| = d(g,e)\), and if \(d\) is also right-invariant then this norm will be conjugation invariant. Thus, we have:

\begin{prop}\label{prop:norm_equiv}
A \(\cn{Met}\)-group is exactly a group equipped with a conjugation-invariant norm. \qed
\end{prop}

Conjugation-invariant norms and their associated metrics have been well studied since work of Klee in the 1950s. Klee proved that every abelian group which is Hausdorff, topologically complete and metrizable admits a bi-invariant metric  \cite{Klee1952}. The existence and properties of such metrics on interesting classes of groups remains an active topic of research due to its significance in fields such as symplectic topology and contact geometry \cite{Han2009, Sandon2015} and in rigidity theory \cite{PolterovichEtal2023}.

In a sense the simplest sort of example is the `word norm' determined by a subset \(S \subset G\) that normally generates \(G\) (meaning that the subgroup generated by conjugates of elements in \(S\) is \(G\)). The norm of \(g \in G\) with respect to \(S\) is the length of the shortest word expressing \(g\) in conjugates of elements of \(S\). Such norms are always integer-valued, but this is not the case in general: examples of `fine' conjugation-invariant norms---those which can take values arbitrarily close to zero---are discussed in \cite{BuragoEtal2008}, Examples 1.19 and 1.21.
\end{eg}

\begin{rmk}\label{rmk:internal_met}
A bi-invariant metric on a group \(G\) is symmetric (thus, a classical metric) if and only if the function \(G \to G\) taking each element to its inverse is a short map. Even in this case, though, \(G\) cannot be interpreted as an internal group object in \(\cn{Met}\) (which is not a cartesian category). The issue is that the function \(G \to G \times G\) which \emph{pairs} each element with its inverse is {not} a short map unless \(G\) is the trivial group.
\end{rmk}


\section{Magnitude homology of enriched groups}\label{sec:MH_groups}

In this section we analyse the iterated magnitude homology, in low degrees, of each of the three classes of examples described in \Cref{sec:enriched_groups}: strict 2-groups, partially ordered groups, and any group with a conjugation-invariant norm.

In each case, a coefficient functor on the base category has already been chosen: on \(\cn{Set}\), we take \(\Sigma\) to be the free abelian group functor; to obtain a coefficient functor on \(\cn{Truth}\) we restrict \(\Sigma\) along the inclusion \(\cn{Truth} \hookrightarrow \cn{Set}\); and on \([0, \infty]\) we take as coefficients the functor \(\Sigma^*\) described in \Cref{def:AbR}. Since the coefficients and the categorical dimension are understood, we will simplify the notation by omitting the superscript \(2\Sigma\), writing \(MH_\bullet(\cl{G})\) for the iterated magnitude homology of a \(\cn{Cat}\)- or \(\cn{PreOrd}\)-group \(\cl{G}\), and \(MH_\bullet^*(\cl{G})\) for the iterated magnitude homology of a \(\cn{Met}\)-group.

We begin by considering strict 2-groups, as in this case the equivalence with crossed modules allows for an especially efficient analysis. Baues and Ellis have defined the (co)homology of a crossed module \(H \xto{\phi} G\) to be the (co)homology of the space \(\bb{B}(H \xto{\phi} G)\) and have studied its properties \cite{BauesCombinatorial1991,EllisHomology1992}. From Ellis's description on pages 5--6 of \cite{EllisHomology1992} it is clear that the space \(\bb{B}(H \xto{\phi} G)\) coincides with the classifying space of the corresponding strict 2-group. Hence, we have:

\begin{thm}\label{thm:MH_BauesEllis}
The iterated magnitude homology of a strict 2-group \(\cl{G}\) is the Baues--Ellis homology of the corresponding crossed module. \qed
\end{thm}

Because the homotopy theory of crossed modules is completely understood, the low-degree homology groups of a strict 2-group are easy to describe.

\begin{eg}\label{eg:2groups_MH}
Given a group \(G\) with a normal subgroup \(N\), let \(\cl{G}_N\) denote the strict 2-group whose construction is described in \Cref{eg:2groups}. By the discussion in that example, the space \(\bb{B}(N \hookrightarrow G)\) has fundamental group \(G/N\). The Hurewicz theorem implies, then, that \(MH_1(\cl{G}_N)\) is the abelianization \((G/N)_\mathrm{ab}\), and applying Proposition 5(iv) in Ellis \cite{EllisHomology1992} we see that \(MH_0(\cl{G}_N) = \bb{Z}\).

More generally, if \(\cl{G}\) is any strict 2-group then, by Ellis's Proposition 5(iv) and the discussion in \Cref{eg:2groups}, we have
\[MH_0(\cl{G}) = \bb{Z} \text{ and } MH_1(\cl{G}) = (G/E)_\mathrm{ab}\]
where \(E \subseteq G\) is the subgroup of objects which admit an arrow from \(e\).
\end{eg}

Not all interesting \(\cn{Cat}\)-groups are strict 2-groups, though, as Example \ref{eg:POgroups} shows. We would like to extend \Cref{eg:2groups_MH} to describe the first two iterated magnitude homology groups of any \(\cn{Cat}\)-group---for instance, any partially ordered group. The next definition and lemma will help make sense of the description.

\begin{defn}\label{def:comp_cat}
The set of \textbf{connected components} in a category \(\cl{X}\) is the set of equivalence classes in \(\ob(\cl{X})\) under the equivalence relation generated by the preorder `\(x \leq y \text{ if } \cl{X}(x,y) \neq \emptyset\)'.
\end{defn}

\begin{lem}\label{lem:CatGp_cong}
Let \(\cl{G}\) be a \(\cn{Cat}\)-group, and \(G\) its underlying group. The connected component of the identity in \(\cl{G}\) is a normal subgroup of \(G\).
\end{lem}

\begin{proof}
Let \(E\) denote the connected component of the identity element \(e\); that is, \(E = \{g \in G \mid g \sim e\}\) where \(\sim\) is the equivalence relation generated by \(\leq\). To see that \(E\) is closed under conjugation, take any \(g \in E\) and \(h \in {G}\). By definition of \(E\), there must exist a zig-zag of 2-morphisms connecting \(g\) to \(e\)---say
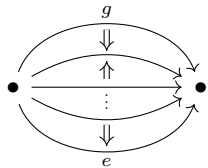
\bdiag[row sep=tiny]
& {} \arrow[Rightarrow, shorten >= -.4em]{d} \\
& {}  \\
\bullet 
& {} 	\arrow[Rightarrow, shorten >= -0em]{u}
	\arrow[draw=none, "\raisebox{+2.2ex}{\scalebox{.6}{\vdots}}" description]{d}
& \bullet 
	\arrow[from=ll, bend left=70, "g"]
	\arrow[from=ll, bend left=30]
	\arrow[from=ll]
	\arrow[from=ll, bend right=30]
	\arrow[from=ll, bend right=70, swap, "e"]
	\\
& {} \arrow[Rightarrow, shorten <= -.4em]{d} \\
& {} 
\ediag
Whiskering each 2-morphism in this zig-zag on the left by \(h^{-1}\) and on the right by \(h\) supplies a zig-zag of 2-morphisms connecting \(h^{-1} g h\) to \(h^{-1}eh = e\), so \(h^{-1}gh \sim e\) and we have \(h^{-1}gh \in E\).

Now take any pair of elements \(g,h \in E\), and take any zig-zag of 2-morphisms connecting \(g\) to \(e\). Whiskering on the right by \(h\) supplies a zig-zag of 2-morphisms connecting \(gh\) to \(h\), which implies that \(gh \sim h\) in \({G}\). Since \(h \sim e\), transitivity ensures we have \(gh \sim e\) and thus \(gh \in E\).

Finally, take any \(g \in E\) and any zig-zag of 2-morphisms connecting \(g\) to \(e\), and whisker on the right by \(g^{-1}\). This supplies a zig-zag of 2-morphisms connecting \(e\) to \(g^{-1}\), so \(g^{-1} \sim e\), which tells us that \(g^{-1} \in E\).
\end{proof}

\Cref{lem:CatGp_cong} implies that the equivalence relation defining the set of connected components is a congruence on the group \({G}\). This tells us that the set of connected components is itself a group, under the operation inherited from \({G}\).

\begin{defn}
Let \(\cl{G}\) be a \(\cn{Cat}\)-group, \(G\) its underlying group, and \(\sim\) the congruence generated by the preorder `\(g \leq h\) if \(\cl{G}(g,h) \neq \emptyset\)'. The group \({G} / \sim\) is the \textbf{group of connected components} in \(\cl{G}\). We will denote it \(\mathrm{Con}(\cl{G})\).
\end{defn}

With this, we can describe the iterated magnitude homology in degrees 0 and 1 of any \(\cn{Cat}\)-group.

\begin{thm}\label{thm:MH_quotient}
Let \(\cl{G}\) be a \(\cn{Cat}\)-group. Then
\[MH_0(\cl{G}) = \bb{Z} \text{ and } MH_1(\cl{G}) = \mathrm{Con}(\cl{G})_{\mathrm{ab}}\]
the abelianization of the group of connected components in \(\cl{G}\). \qed
\end{thm}

    \begin{lrbox}{\OneSingleA}%
        \begin{tikzcd}%
            \bullet \arrow[]{r}{g_0} & \bullet
        \end{tikzcd}%
    \end{lrbox}
    
    \begin{lrbox}{\OneSingleB}%
        \begin{tikzcd}%
            \bullet \arrow[]{r}{h_0} & \bullet
        \end{tikzcd}%
    \end{lrbox}
    
    \begin{lrbox}{\TwoSinglesA}%
        \begin{tikzcd}%
            \bullet \arrow[]{r}{g_0} & \bullet \arrow[]{r}{g_1} & \bullet
        \end{tikzcd}%
    \end{lrbox}
    
    \begin{lrbox}{\TwoSinglesB}%
        \begin{tikzcd}%
            \bullet \arrow[]{r}{h_0} & \bullet \arrow[]{r}{h_1} & \bullet
        \end{tikzcd}%
    \end{lrbox}
    
    \begin{lrbox}{\OnePairA}%
        \begin{tikzcd}[row sep=tiny, column sep=tiny]%
            & {} \arrow[Rightarrow, shift right=1, "\alpha"]{dd} \\ [-5]
            \bullet \arrow[bend left=35]{rr}{g_0} \arrow[bend right=35, swap]{rr}{h_0} & & \bullet \\ [-5]
            & {}
        \end{tikzcd}%
    \end{lrbox}

    \begin{lrbox}{\TwoPairs}%
        \begin{tikzcd}[row sep=tiny, column sep=tiny]%
            & {} \arrow[Rightarrow, shift right=1, "\alpha_1"]{dd} && \arrow[Rightarrow, shift right=1, "\alpha_2"]{dd} \\ [-5]
            \bullet \arrow[bend left=35]{rr}{g_0} \arrow[bend right=35, swap]{rr}{h_0} && \bullet \arrow[bend left=35]{rr}{g_1} \arrow[bend right=35, swap]{rr}{h_1} && \bullet \\ [-5]
            & {} && {}
        \end{tikzcd}%
    \end{lrbox}

\begin{proof}
By the Eilenberg--Zilber theorem, we can compute the magnitude homology of \(\cl{G}\) as the total homology of the double complex \(C = C(MB^{MB^\Sigma}(\cl{G}))\). We will study the total homology via the column filtration spectral sequence \(\{E_{\bullet\bullet}^r\}\) of \(C\), using the isomorphism \(H_0(\Tot(C)) \cong E_{00}^\infty\) and the exact sequence
\begin{equation}\label{eq:MH_quotient_ses}
0 \to E_{01}^\infty \to H_1(\Tot(C)) \to E_{10}^\infty \to 0
\end{equation}
whose existence is implied by the convergence of \(\{E_{\bullet\bullet}^r\}\) to \(H_\bullet(\Tot(C))\). (For explanation of this see, e.g., the proof of Lemma C.3.2 in \cite{Roff2022}.) This lets us restrict attention to the southwest corner of \(C\), which has the form
\bdiag
\vdots \arrow{d}{\partial^v_{10}} & \vdots \arrow{d} & \vdots \arrow{d}\\
\bb{Z}\cdot\left\{\bullet \right\} \arrow{d}{\partial^v_{00}} & \bb{Z} \cdot \left\{ \usebox{\OnePairA} \right\} \arrow{d}{\partial^v_{01}} \arrow[swap]{l}{\partial^h_{10}} & \bb{Z} \cdot \left\{ \usebox{\TwoPairs} \right\} \arrow{d}{\partial^v_{11}} \arrow[swap]{l}{\partial^h_{11}} & \cdots \arrow{l} \\
\bb{Z}\cdot\left\{\bullet \right\} & \bb{Z} \cdot \left\{ \usebox{\OneSingleA} \right\} \arrow[swap]{l}{\partial^h_{00}} & \bb{Z} \cdot \left\{ \usebox{\TwoSinglesA} \right\} \arrow[swap]{l}{\partial^h_{01}} & \cdots \arrow{l} 
\ediag
where the horizontal differentials come from the simplicial structure induced by composition of 1-morphisms---that is, by multiplication in the underlying group---and the vertical differentials from the simplicial structure induced by composition of 2-morphisms.

To find page \(E^1\) we take vertical homology. In column 0 we have 
\[\partial^v_{q0} = \sum_{i=0}^{q+1} (-1)^i \id = \begin{cases} 0 & \text{when \(q\) is even} \\
\id &\text{when \(q\) is odd}\end{cases}\]
so \(E_{00}^1 = \bb{Z}\) and \(E_{q0}^1 = 0\) for all \(q > 0\). This implies in particular that \(E_{10}^\infty = 0\), so from (\ref{eq:MH_quotient_ses}) we have \(MH_1(\cl{G}) = H_1(\Tot(C)) \cong E_{01}^\infty = E_{01}^2\), since in a first-quadrant spectral sequence the bottom row stabilizes at page \(E^2\).

Column 1 is isomorphic to the complex of chains in the nerve of the category of elements in \(\cl{G}\). Explicitly,
\[\partial^v_{01}\left( \usebox{\OnePairA} \right) = \usebox{\OneSingleB} - \usebox{\OneSingleA}\]
so in \(E_{01}^1 = C_{01} / \im(\partial^v_{01})\) we have an identification of generators
\[\left[ \usebox{\OneSingleA} \right] = \left[ \usebox{\OneSingleB} \right]\]
whenever there exists a 2-morphism \(g_0 \Rightarrow h_0\) or \(h_0 \Rightarrow g_0\). In other words, \(E_{01}^1\) is the free abelian group on the elements of \(\mathrm{Con}(\cl{G})\). Similarly, we have
\[E_{02}^1 = \bb{Z} \cdot \left\{ \left[ \usebox{\TwoSinglesA} \right] \right\} \]
where
\[ \left[ \usebox{\TwoSinglesA} \right] = \left[ \usebox{\TwoSinglesB} \right] \]
whenever there are 2-morphisms \(g_0 \Rightarrow g_1\) \underline{and} \(h_0 \Rightarrow h_1\), or \(g_1 \Rightarrow g_0\) \underline{and} \(h_1 \Rightarrow h_0\).

The differential on page \(E^1\) is induced on vertical homology by the horizontal differential in \(C\). Since
\[\partial_{00}^h \left( \usebox{\OneSingleA} \right) = \bullet - \bullet = 0\]
we have \(MH_0(\cl{G}) = E_{00}^2 = E_{00}^1 = \bb{Z}\) as claimed, and \(MH_1(\cl{G}) = E_{01}^2 = E_{01}^1 / \im(d)\) where \(d\) is the map induced on vertical homology by \(\partial^h_{01}\). Thus,
\[MH_1(\cl{G}) = \frac{\bb{Z} \cdot \mathrm{Con}(\cl{G})}{\langle [g_0] - [g_0 g_1] + [g_1] \rangle}\]
which is the abelianization of the group \(\mathrm{Con}(\cl{G})\).
\end{proof}

\begin{eg}\label{eg:POgroups_MH}
Given a partially ordered group \(\cl{G} = (G, \leq)\), let \(\widetilde{\cl{G}} = (G, \sim)\) denote the same group equipped with the equivalence relation generated by \(\leq\). The connected components do not change when we pass from \(\cl{G}\) to \(\widetilde{\cl{G}}\). In particular, the connected component of the identity in \(\cl{G}\) is
\[E = \{g \in G \mid g \sim e\} = P_{\widetilde{\cl{G}}},\]
the positive cone of \(\widetilde{\cl{G}}\) as defined in \Cref{eg:POgroups}. Thus, \Cref{thm:MH_quotient} says that
\[MH_0(\cl{G}) = \bb{Z} = MH_0(\widetilde{\cl{G}}) \; \text{ and } \; MH_1(\cl{G}) = \left(\widetilde{\cl{G}}/P_{\widetilde{\cl{G}}}\right)_\mathrm{ab} = MH_1(\widetilde{\cl{G}}),\]
so the first two homology groups do not distinguish \(\widetilde{\cl{G}}\) from \(\cl{G}\).

For readers familiar with the homotopy theory of ordinary categories this may not come as a surprise, since, at least in the case of a \emph{finite} preorder \(Q\), the ordinary classifying space does not distinguish \(Q\) from the equivalence relation \(\widetilde{Q}\) that it generates. Indeed, in the finite case it is not hard to see that the inclusion \(Q \hookrightarrow \widetilde{Q}\) satisfies the conditions of Quillen's Theorem A and so induces a homotopy equivalence on the classifying space \cite{HigherK1973}. Note, however, that the only translation-invariant partial order on a \emph{finite} group is the trivial one (for an explanation, see \cite{Patotski2014}), so the interesting examples are necessarily infinite. In the case of an \emph{infinite} poset it is less clear whether Quillen's Theorem A applies.
\end{eg}

\Cref{eg:POgroups_MH} suggests that iterated magnitude homology may be rather insensitive as an invariant of partially ordered groups, and preordered groups more generally. By contrast, in the case of a group with a conjugation-invariant norm we will find evidence that iterated magnitude homology does capture information not just about the structure of the group itself but about its geometry under the norm. We begin by describing the double magnitude nerve, which requires the following definition.

\begin{defn}
Let \(\cl{G} = (G, |-|)\) be a group with a conjugation-invariant norm, and let \(d\) be the associated metric on \(G\). Given a matrix
\[\vec{g} = \begin{bmatrix}
g_{00} & \cdots & g_{0n} \\
\vdots & & \vdots \\
g_{m0} & \cdots & g_{mn}
\end{bmatrix}\]
of elements of \(\cl{G}\), the \textbf{total length} of \(\vec{g}\) is the sum of the lengths of its columns:
\[L(\vec{g}) = \sum_{i=0}^n \sum_{j=0}^{m-1} d(g_{ij}, g_{i,j+1}) = \sum_{i=0}^n \sum_{j=0}^{m-1} |g_{ij} g_{i,j+1}^{-1}|.\]
In particular, if \(m=0\), then \(L(\vec{g}) = 0\). The total length of the empty matrix is \(0\).
\end{defn}

Since the magnitude nerve of a metric space is a real-graded simplicial abelian group, the double magnitude nerve \(B^*_{\bullet\bullet} = MB^{MB^\Sigma}(\cl{G})\) is a real-graded bisimplicial abelian group. Unravelling \Cref{def:2MH} in this case yields the following description.

\begin{prop}\label{prop:double_MB_normed}
In grading \(\ell \in [0, \infty)\), for each \(q \geq 0\) and \(p >0\), the abelian group \(B_{pq}^\ell\) is freely generated by \((q+1) \times p\) matrices of elements of \(G\) whose total length is \(\ell\). Meanwhile, \(B_{0q}^\ell = 0\) for all \(\ell > 0\), and \(B_{0q}^0 = \mathbb{Z}\). The horizontal face maps \(\delta_j^h\), which are defined when \(p>0\), are specified on generators as follows:
\begin{itemize}
\item The map \(\delta_0^h\) deletes the first column of a generating matrix \(\vec{g}\) provided the length of this column is \(0\), which holds if and only if all its entries are equal. Otherwise, \(\delta_0^h(\vec{g}) = 0\). The map \(\delta_{p}^h\) acts in the same way with respect to the final column of \(\vec{g}\).
\item For \(0 < j < p\) the map \(\delta_j^h\) is induced by multiplication at column \(j-1\)---
\[\delta_i^h \left( \begin{bmatrix}
g_{00} & \cdots & g_{0,p-1} \\
\vdots & & \vdots \\
g_{q0} & \cdots & g_{q,p-1}
\end{bmatrix} \right) = \begin{bmatrix}
g_{00} & \cdots & g_{0,j-1} g_{0j} & \cdots & g_{0,p-1} \\
\vdots & & & & \vdots \\
g_{q0} & \cdots & g_{q,j-1} g_{qj} & \cdots & g_{q,p-1}
\end{bmatrix}
\]
---provided that for every \(0 \leq i < q\) we have
\[d(g_{i,j-1}g_{ij}, g_{i+1,j-1}g_{i+1,j}) = d(g_{i,j-1}, g_{i+1,j-1}) + d(g_{ij},g_{i+1,j}).\]
Otherwise, \(\delta_j^h(\vec{g}) = 0\).
\end{itemize}
When \(p>0\), each horizontal degeneracy map \(\sigma_j^h\) inserts a column of identity elements after column \(j\). When \(p=0\), the unique horizontal degeneracy map \(B_{0q}^0 = \bb{Z} \to B_{1q}^0\) takes the element \(1\) to a column (of length \(q+1\)) of identity elements.

When \(p > 0\), the vertical face maps are specified on generators as follows:
\begin{itemize}
\item The map \(\delta_0^v\) deletes the top row of a generating matrix \(\vec{g}\) provided that \(g_{0j} = g_{1j}\) for \(0 \leq j < p\); otherwise, \(\delta_0^v(\vec{g}) = 0\). The map \(\delta_q^v\) acts in the same way with respect to the bottom row.
\item For \(0 < i < q\), the map \(\delta_i^v\) deletes row \(i\) of \(\vec{g}\) provided that for every \(0 \leq j < p\) we have
\[d(g_{i-1,j}, g_{i+1,j}) = d(g_{i-1,j}, g_{ij}) + d(g_{ij}, g_{i+1,j});\]
otherwise, \(\delta_i^v(\vec{g}) = 0\).
\end{itemize}
Each vertical degeneracy map \(\sigma_i^v\) inserts a duplicate of row \(i\). When \(p=0\), all the vertical face and degeneracy maps are identities.
\qed
\end{prop}

Recall from \Cref{sec:MH} that points \(x\) and \(y\) in a metric space are said to be \emph{non-adjacent} if there exists a point \(w\) lying {strictly between} \(x\) and \(y\) in the sense that \(w \neq x,y\) and \(d(x,y) = d(x,w) + d(w,y)\). Otherwise, \(x\) and \(y\) are said to be \emph{adjacent}. In certain cases, the elements in a normed group \(G\) which are adjacent to the identity have special significance: they can be thought of as `primitive' with respect to the norm on \(G\). For instance, if the norm is a word norm with respect to some set \(S\) which normally generates \(G\) (as described in \Cref{eg:Metgroups}), then the elements adjacent to the identity are precisely those which are conjugate to some element of \(S\). We will call such elements \emph{indecomposable}.

\begin{defn}
Let \(\cl{G} = (G, |-|)\) be a normed group. An element \(g \in G\) will be called \textbf{indecomposable} if \(g \neq e\) and for all \(h \in G \backslash \{e, g\}\) we have
\[|g| < |h^{-1}g| + |h|.\]
In terms of the associated metric on \(G\), an element is indecomposable if and only if it is adjacent to the identity.
\end{defn}

\Cref{thm:MH_normed_gps} says that if \(\cl{G} = (G,|-|)\) is a group with a conjugation-invariant norm, then iterated magnitude homology recovers the ordinary group homology of \(G\) \emph{and} counts (up to conjugacy) its indecomposable elements of every length. The proof, which is similar in strategy to that of \Cref{thm:MH_quotient} but more involved, is deferred to Appendix \ref{sec:proofs}.

\begin{thm}\label{thm:MH_normed_gps}
Let \(\cl{G} = (G, |-|)\) be a group with a conjugation-invariant norm. In length grading \(0\), iterated magnitude homology recovers the group homology of \(G\):
\[MH_\bullet^0(\cl{G}) \cong H_\bullet(G).\]
In length gradings \(\ell > 0\) we have \(MH_0^\ell(\cl{G}) = 0 = MH_1^\ell(\cl{G})\), while
\[MH_2^\ell(\cl{G}) \cong \bb{Z} \cdot \left\{ \begin{array}{l} \text{conjugacy classes of indecomposable} \\ \text{elements of norm \(\ell\) in \(G\)} \end{array} \right\}. \tag*{\qed}\]
\end{thm}


\section{Magnitude homology of strict \(n\)-categories}\label{sec:MH_ncats}

Since the magnitude nerve functor with respect to a strong symmetric monoidal size functor is always strong symmetric monoidal, this is in particular true when we take the size functor to be \(MB^\Sigma\) itself. That is, the double magnitude nerve
\[MB^{MB^{\Sigma}}: \cf{V}\cn{Cat} \to [\Delta^\op \times \Delta^\op, \ab{A}]\]
is strong symmetric monoidal. As the functor \(\diag: [\Delta^\op \times \Delta^\op, \ab{A}] \to  [\Delta^\op, \ab{A}]\) is strong symmetric monoidal too, so is their composite \(MB^{2\Sigma}: \cf{V}\cn{Cat} \to [\Delta^\op, \ab{A}]\). This suggests going a step further, to define a magnitude nerve for categories enriched in the category of \(\cf{V}\cn{Cat}\)-categories: 
\[MB^{3\Sigma}(-) = \diag \: MB^{MB^2}(-).\]
We will pursue this idea only in the case that \(\cf{V} = \cn{Set}\), where repeated iteration allows us to define a {sequence} of homology theories: a magnitude homology theory on the category of strict \(n\)-categories, for every \(n \geq 1\).

\paragraph{Conventions}
Throughout this section, `\(n\)-category' will mean `strict \(n\)-category'. Since we consider only small enriched categories in this paper, an \(n\)-category will be assumed to have a set of objects. We refer to the objects of an \(n\)-category as \emph{0-cells}, and its \(k\)-morphisms as \emph{\(k\)-cells}. The category of \(n\)-categories is denoted by \(n\cn{Cat}\).

\begin{defn}\label{def:nMH}
Let \(\Sigma: \cn{Set} \to \cn{Ab}\) be the free abelian group functor. The \textbf{magnitude nerve} functor on the category of \(n\)-categories is defined as follows.
\begin{itemize}
\item Set \(MB^1 = MB^\Sigma: \cn{Cat} \to [\Delta^\op, \cl{Ab}]\).
\item For each \(n > 1\), define
\[MB^n = \diag\left(MB^{MB^{n-1}}(-)\right) : n\cn{Cat} \to [\Delta^\op, \cl{Ab}].\]
\end{itemize}
The \textbf{magnitude chain complex} of an \(n\)-category \(\cl{X}\) is
\[MC^n_\bullet(\cl{X}) = C(MB^n(\cl{X})),\]
and the \textbf{magnitude homology} of \(\cl{X}\) is
\[MH^n_\bullet(\cl{X}) = H_\bullet(MC^n_\bullet(\cl{X})).\]
\end{defn}

We will illustrate the magnitude homology of \(n\)-categories by drawing out a certain pleasing feature---namely, its behaviour with respect to an operation analogous to the two-point suspension of topological spaces.

\begin{defn}
Given an \(n\)-category \(\cl{X}\), the \textbf{suspension} of \(\cl{X}\) is an \((n+1)\)-category \(\Gamma\cl{X}\) with two objects, \(A\) and \(B\); the \(n\)-category \(\Gamma\cl{X}(A,B)\) is equivalent to \(\cl{X}\), while \(\Gamma\cl{X}(B,A)\) is empty and the other two hom objects are given by the terminal \(n\)-category.
\end{defn}

\begin{figure}[h]
\adjustbox{scale=1.2, center}{%
\begin{tikzcd}[ampersand replacement=\&, row sep=10, column sep=25,
execute at end picture={
	\draw[gray] (0.25,0) node[left] {\(\cl{X}\)};
  }]
\& \cdot 
	\arrow[shorten <= -0.7em, shorten >= -0.7em]{dr} 
	\arrow[no head, shorten <= -0.7em, shorten >= -0.7em, bend right=40]{dd}
	\arrow[no head, dashed, shorten <= -0.7em, shorten >= -0.7em, bend left=30]{dd}
	\\
|[alias=dot]| A \bullet 
	\arrow[no head, shorten <= -0.7em, shorten >= -0.7em]{ur} 
	\arrow[no head, shorten <= -0.7em, shorten >= -0.7em]{dr} \& 
\&
|[alias=star]| \bullet B
\\
\& \cdot
	\arrow[shorten <= -0.7em, shorten >= -0.7em]{ur} 
\end{tikzcd}
}
\end{figure}

Below, we will see that magnitude homology behaves with respect to this form of suspension very much as unreduced singular homology behaves with respect to the two-point suspension of topological spaces.

The topological story goes as follows. Given a nonempty space \(Y\), the \emph{suspension} of \(Y\), denoted \(SY\), is the quotient of \(Y \times [0, 1]\) obtained by collapsing \(Y \times \{0\}\) to one point and collapsing \(Y \times \{1\}\) to another point. After decomposing \(SY\) as a union of two cones on \(Y\), a straightforward application of the Mayer--Vietoris sequence reveals that
\[H_k(SY) \cong \begin{cases} \bb{Z} & k=0 \\
H_{k-1}(Y) & k \geq 2 \end{cases}\]
while \(H_1(SY) \oplus \bb{Z} \cong H_0(Y)\).

Our theorem is formally very similar to the topological one:

\begin{thm}\label{thm:MH_suspension}
Let \(\cl{X}\) be a nonempty strict \(n\)-category. Then
\[MH_k^{n+1}(\Gamma\cl{X}) \cong \begin{cases} \bb{Z} & k=0 \\
MH_{k-1}^{n}(\cl{X}) & k \geq 2 \end{cases}\]
while
\[MH_1^{n+1}(\Gamma\cl{X}) \oplus \bb{Z} \cong MH_0^n(\cl{X}).\]
\end{thm}

Unlike the classical statement, which expresses a property of a single homology theory, ours expresses a relationship between members of a family of homology theories (namely the family \((MH_\bullet^n)_{n \geq 1}\)). Accordingly, the proof of \Cref{thm:MH_suspension}, which occupies the remainder of this section, is very different from the standard proof of the topological result. Rather than decomposing the \((n+1)\)-category \(\Gamma\cl{X}\) into smaller pieces, we compute its homology directly, as the total homology of the double complex \(C(MB^{MB^n}(\Gamma\cl{X}))\). 

We begin by describing the \(0^\th\) magnitude homology group of any \(n\)-category---not necessarily one obtained by suspension. The \(0^\th\) homology group captures the same information irrespective of the categorical dimension.

\begin{defn}
The set of \textbf{connected components} in an \(n\)-category \(\cl{X}\) is the set of equivalence classes in \(\ob(\cl{X})\) under the equivalence relation generated by the preorder `\(x \leq y \text{ if } \cl{X}(x,y)\) is nonempty'. We denote it by \(\mathrm{Con}(\cl{X})\).
\end{defn}

\begin{prop}\label{prop:MH_0Pi_0}
Let \(\cl{X}\) be an \(n\)-category. Then \(MH_0^n(\cl{X}) = \bb{Z} \cdot \mathrm{Con}(\cl{X})\).
\end{prop}

The proof comes down to a lemma describing the first two groups in the iterated magnitude nerve and the face maps between them; \Cref{prop:MH_0Pi_0} follows on taking homology of the associated chain complex. To state the lemma, we note that for \(k >0\), every \(k\)-cell in an \(n\)-category \(\cl{X}\) is a \((k-1)\)-cell in exactly one hom-category, say \(\cl{X}(x,y)\). We call \(x\) the \textbf{domain \(0\)-cell} of \(k\), and \(y\) its \textbf{codomain \(0\)-cell}.

\begin{lem}\label{lem:MB^n}
For any \(n\)-category \(\cl{X}\) we have
\[MB_0^n(\cl{X}) = \bb{Z} \cdot \{ 0 \text{-cells in } \cl{X} \} \; \text{ and } \; MB_1^n(\cl{X}) = \bb{Z} \cdot \{n \text{-cells in } \cl{X}\}.\]
Each \(n\)-cell is mapped by \(\delta_0\) to its codomain \(0\)-cell and by \(\delta_1\) to its domain \(0\)-cell.
\end{lem}

\begin{proof}
The statement holds for \(n=1\), by definition of the magnitude nerve of an ordinary category. Let \(n>1\) and suppose the statement holds for \(n-1\). We have
\[MB_{0\bullet}^{MB^{n-1}}(\cl{X}) = \bigoplus_{x \in \ob (\cl{X})} I_\pw,\]
where \(I_\pw\) is unit for the pointwise tensor product in \([\Delta^\op, \cn{Ab}]\): it has exactly one copy of \(\bb{Z}\) in every degree, and all its face and degeneracy maps are identities. Thus, for each \(k \geq 0\) we have \(MB_{0k}^{MB^{n-1}}(\cl{X}) = \bb{Z}\cdot \ob(\cl{X}) = \bb{Z} \cdot \{0\text{-cells in } \cl{X}\}\).

Meanwhile
\[MB_{1\bullet}^{MB^{n-1}}(\cl{X}) = \bigoplus_{x,y \in \ob (\cl{X})} MB_\bullet^{n-1}(\cl{X}(x,y))\]
and so
\begin{align}
MB_{10}^{MB^{n-1}}(\cl{X}) \nonumber &= \bigoplus_{x,y \in \ob(\cl{X})} MB_0^{n-1}(\cl{X}(x,y)) \nonumber\\
&= \bigoplus_{x,y \in \ob(\cl{X})} \bb{Z} \cdot \{0\text{-cells in }\cl{X}(x,y)\} \label{ind1}\\
&= \bb{Z}\cdot\{1\text{-cells in } \cl{X}\} \nonumber
\end{align}
where (\ref{ind1}) uses the inductive assumption. Also
\begin{align}
MB_{11}^{MB^{n-1}}(\cl{X}) \nonumber &= \bigoplus_{x,y \in \ob(\cl{X})} MB_1^{n-1}(\cl{X}(x,y)) \nonumber\\
&= \bigoplus_{x,y \in \ob(\cl{X})} \bb{Z} \cdot \{(n-1)\text{-cells in }\cl{X}(x,y)\} \label{ind2}\\
&= \bb{Z} \cdot \{n\text{-cells in } \cl{X}\} \nonumber
\end{align}
using the inductive assumption again in (\ref{ind2}).

So the southwest corner of \(MB^{MB^{n-1}}(\cl{X})\) has the form
\bdiag
\bb{Z} \cdot \{0\text{-cells in } \cl{X}\}
	\arrow[shift left=2]{d}
	\arrow[shift right=2, swap]{d}
& \bb{Z} \cdot \{n\text{-cells in } \cl{X}\}
	\arrow[shift left=1.7]{l}
	\arrow[shift right=1.7, swap]{l}
	\arrow[shift left=2]{d}{\delta^v_1} 
	\arrow[shift right=2, swap]{d}{\delta^v_0}
	\\
\bb{Z} \cdot \{0\text{-cells in } \cl{X}\}
& \bb{Z} \cdot \{1\text{-cells in } \cl{X}\}
	\arrow[shift left=1.7]{l}{\delta^h_1} 
	\arrow[shift right=1.7, swap]{l}{\delta^h_0}
\ediag
where \(\delta_0^v\) is induced by mapping each \(n\)-cell to {its codomain 1-cell}, and \(\delta_0^h\) by mapping each 1-cell to {its codomain 0-cell}; \(\delta_1^v\) takes an \(n\)-cell to its {domain 1-cell}, and \(\delta_1^h\) takes a 1-cell to its {domain 0-cell}. The diagonal face map \(\delta_0 = \delta_0^h \circ \delta_0^v\), then, maps each \(n\)-cell \(\alpha\) to {its codomain 0-cell}, while \(\delta_1\) maps \(\alpha\) to its domain 0-cell, as claimed.
\end{proof}

\begin{proof}[Proof of \Cref{prop:MH_0Pi_0}]
By \Cref{lem:MB^n}, taking zeroth homology in \(MC^n(\cl{X})\) has the effect of quotienting the generators of \(MB_0^n(\cl{X})\)---the set of \(0\)-cells in \(\cl{X}\)---by the equivalence relation generated by ``\(x \sim y\) if there exists an \(n\)-cell in \(\cl{X}\) whose domain 0-cell is \(x\) and whose codomain 0-cell is \(y\)''. That is, \(x \sim y\) if the \((n-1)\)-category \(\cl{X}(x,y)\) contains at least one \((n-1)\)-cell. But an \((n-1)\)-category contains an \((n-1)\)-cell if and only if it contains at least one \(0\)-cell; that is, if and only if it is nonempty.
\end{proof}

In the proof of \Cref{thm:MH_suspension}, it will simplify matters to replace a double complex by one in which the rows have been normalized. The next lemma explains how this can be done.

Any bisimplicial object \(B\) in an abelian category \(\ab{A}\) can be regarded (in either of two ways) as a simplicial object of \([\Delta^\op, \ab{A}]\), so we can take its normalized chain complex to obtain an object \(NB\) of \(\Ch([\Delta^\op, \ab{A}])\). Now applying the unnormalized chain complex functor in every degree yields an object of \(\Ch(\Ch(\ab{A}))\). Equivalently, this is the double complex obtained from \(B\) by taking normalized complexes in every row and unnormalized complexes in every column (or vice versa). We will denote this double complex by \(C^v N^h B\).

\begin{lem}\label{lem:norm_rows}
Let \(B\) be a bisimplicial object in an abelian category. Then 
\[H_\bullet\Tot \: CB \cong H_\bullet\Tot \: C^v N^h B.\]
\end{lem}

\begin{proof}
Taking quotients by the horizontal degeneracies in \(B\) defines a map of double complexes \(CB \to C^v N^h B\). As the quotient map in each row is a quasi-isomorphism, the induced morphism of row-filtration spectral sequences is an isomorphism on page \(E^1 = H^h CB\), and thus induces an isomorphism on total homology. (See, e.g., Brown \cite{BrownCohomology1982}, VII, Proposition 2.6.)
\end{proof}

With this, we can prove the main theorem of this section, which is the final theorem of the paper. For convenience, we repeat the statement.

\begin{thmMH_suspension}
Let \(\cl{X}\) be a nonempty strict \(n\)-category. Then
\[MH_k^{n+1}(\Gamma\cl{X}) \cong \begin{cases} \bb{Z} & k=0 \\
MH_{k-1}^{n}(\cl{X}) & k \geq 2 \end{cases}\]
while
\[MH_1^{n+1}(\Gamma\cl{X}) \oplus \bb{Z} \cong MH_0^n(\cl{X}).\]
\end{thmMH_suspension}

\begin{proof}
By the Eilenberg--Zilber theorem there is a natural isomorphism
\[MH_\bullet^{n+1}(\Gamma \cl{X}) \cong H_\bullet \Tot \: C (MB^{MB^{n}}(\Gamma\cl{X}))\]
and by \Cref{lem:norm_rows} the right hand side is isomorphic to \(H_\bullet\Tot\:C\), where
\[C = C^v N^h (MB^{MB^{n}}(\Gamma\cl{X})) = C^v (MC^{MB^{n}}(\Gamma\cl{X})).\]
We begin by describing the double complex \(C\). As \(\Gamma\cl{X}\) has two objects,
\[MC_0^{MB^{n}}(\Gamma\cl{X}) = I_\pw \oplus I_\pw\]
where \(I_\pw\) is the pointwise tensor unit in \([\Delta^\op, \cn{Ab}]\). (Note that this is an equation of simplicial objects.) Meanwhile, applying \Cref{lem:normalized}, we see that
\begin{align*}
MC_1^{MB^{n}}(\Gamma\cl{X}) &= MB_\bullet^{n}(\Gamma\cl{X}(A,B)) \oplus MB_\bullet^{n}(\Gamma\cl{X}(B,A)) \\
&= MB_\bullet^{n}(\Gamma\cl{X}(A,B)) \oplus 0 \\
&\cong MB_\bullet^{n}(\cl{X})
\end{align*}
as \(\Gamma\cl{X}(B,A)\) is empty and \(\Gamma\cl{X}(A,B) \simeq \cl{X}\).

Since, for each \(k \geq 2\), every summand in the \(k^\th\) chain group involves a factor of \(MB_\bullet^{n}(\Gamma\cl{X}(B,A)) = 0\), we have \(MC_k^{MB^{n}}(\Gamma\cl{X}) = 0\) for \(k \geq 2\). Hence,
\[MC_k^{MB^{n}}(\Gamma\cl{X}) = \begin{cases}
I_\pw \oplus I_\pw & k=0 \\
MB_\bullet^{n}(\cl{X}) & k=1 \\
0 & k > 1
\end{cases}\]
and the only nonzero map in the differential is
\[\partial_{1\bullet} : MB_\bullet^{n}(\cl{X}) \to I_\pw \oplus I_\pw \cong MB_\bullet^{n}(\cl{1}) \oplus MB_\bullet^{n}(\cl{1}).\]
This map is given by \(\delta_1^0 - \delta_1^1\), where \(\delta_1^0\) is induced by the terminal map to the second summand, and \(\delta_1^1\) by the terminal map to the first. In summary, \(MC_\bullet^{MB^{n}}(\Gamma\cl{X}) \in \Ch([\Delta^\op, \cl{Ab}])\) looks like this---
\bdiag
\vdots & \vdots & \vdots \\ [-2em]
\bb{Z}^2 
	\arrow[shift left=1em]{d} \arrow[]{d} \arrow[shift right=1em]{d} & 
MB_2^{n}(\cl{X})
	\arrow{l}
	\arrow[shift left=1em]{d} \arrow[]{d} \arrow[shift right=1em]{d} & 
0
	\arrow{l}
	\arrow[shift left=1em]{d} \arrow[]{d} \arrow[shift right=1em]{d} &
\cdots
	\arrow{l}  \\
\bb{Z}^2 
	\arrow[shift left = 0.5em]{d} \arrow[shift right = 0.5em]{d}
	\arrow[shift left=0.5em]{u} \arrow[shift right=0.5em]{u} &
MB_1^{n}(\cl{X})
	\arrow{l} 
	\arrow[shift left=0.5em]{u} \arrow[shift right=0.5em]{u} 
	\arrow[shift left = 0.5em]{d} \arrow[shift right = 0.5em]{d} & 
0 
	\arrow{l} 
	\arrow[shift left = 0.5em]{d} \arrow[shift right = 0.5em]{d}
	\arrow[shift left=0.5em]{u} \arrow[shift right=0.5em]{u} &
\cdots
	\arrow{l}  \\
\bb{Z}^2 
	\arrow{u} &
MB_0^{n}(\cl{X})
	\arrow{l} \arrow{u} &
0
	 \arrow{l} \arrow{u} &
\cdots
	 \arrow{l}
\ediag
---and now taking the unnormalized complex in each column yields the double complex \(C\):
\bdiag
\vdots \arrow{d}{0} & \vdots \arrow{d} & \vdots \arrow{d} \\
\bb{Z}^2  \arrow{d}{\id} & \arrow{l}{\partial_{12}} \widetilde{MC}_2^{n}(\cl{X}) \arrow{d} & \arrow{l} 0 \arrow{d} & \arrow{l} \cdots  \\
\bb{Z}^2  \arrow{d}{0} & \arrow{l}{\partial_{11}} \widetilde{MC}_1^{n}(\cl{X}) \arrow{d} & \arrow{l} 0 \arrow{d} & \arrow{l} \cdots  \\
\bb{Z}^2 & \arrow{l}{\partial_{10}} \widetilde{MC}_0^{n}(\cl{X}) & \arrow{l} 0 & \arrow{l} \cdots
\ediag

As \(MH_k^{n+1}(\Gamma \cl{X}) \cong H_k\Tot\: C\) and \(C\) is concentrated in the first two columns, there is for each \(k\) an exact sequence
\[0 \to H^h_1 H^v_{k-1} C \to MH_k^{n+1}(\Gamma \cl{X}) \to H^h_0 H^v_k C \to 0.\]
(See, for instance, Exercise 5.2.1 of Weibel \cite{WeibelIntroduction1994}.) But \(H^v_k C_{0\bullet} = 0\) for \(k > 0\), so this gives isomorphisms
\begin{equation}\label{eq:HhHv}
MH_k^{n+1}(\Gamma \cl{X}) \cong H^h_1 H^v_{k-1} C
\end{equation}
whenever \(k > 0\), while \Cref{prop:MH_0Pi_0} tells us that
\[
MH_0^{n+1}(\Gamma \cl{X}) \cong \bb{Z}\cdot \mathrm{Con}(\Gamma\cl{X}) = \bb{Z}.
\]
Moreover, after taking vertical homology, the induced horizontal differential is zero in all rows other than row 0, so (\ref{eq:HhHv}) says
\begin{equation}\label{eq:H_k}
MH_k^{n+1}(\Gamma \cl{X}) \cong H^v_{k-1}C_{1\bullet} =\begin{cases}
 MH_{k-1}^n(\cl{X}) & k > 1 \\
 \ker(H^v \partial_{10}) & k = 1.\end{cases}
 \end{equation}
Finally, the homomorphism \(H^v \partial_{10}: MH_0^{n}(\cl{X}) = \bb{Z}\cdot\mathrm{Con}(\cl{X}) \to \bb{Z}^2\) maps all generators to the element \((-1,1)\), spanning a single copy of \(\bb{Z}\). From the second case of (\ref{eq:H_k}), then, we get
\[MH_0^n(\cl{X}) \cong \ker(H^v \partial_{10}) \oplus \im (H^v \partial_{10}) \cong MH_1^{n+1}(\Gamma \cl{X}) \oplus \bb{Z}. \qedhere\]
\end{proof}

\begin{eg}[Homology of `spheres']
Let \(\bb{S}^0\) denote a two-element set and for each \(n > 0\) let \(\bb{S}^n\) denote the \(n\)-category with two objects and two parallel \(k\)-cells in every dimension \(0 < k \leq n\); then for each \(n > 0\) the \(n\)-category \(\bb{S}^n\) is the suspension of \(\bb{S}^{n-1}\). For instance, \(\bb{S}^2\) looks like this:
\bdiag
& {} \arrow[Rightarrow, bend right=50, shorten >=1.5ex, shorten <=1.5ex]{dd} \arrow[Rightarrow, bend left=50, shorten >=1.5ex, shorten <=1.5ex]{dd} \\
\bullet \arrow[bend left=85]{rr} \arrow[swap, bend right=85]{rr} & & \bullet \\
& {}
\ediag
In \Cref{eg:S1} we saw that \(MH_k^\Sigma(\bb{S}^1) = \bb{Z}\) for \(k=0\) or \(1\), while in all higher degrees the magnitude homology vanishes. Applying \Cref{thm:MH_suspension} we find by induction that for every \(n \in \bb{N}\) we have
\[MH_k^n(\bb{S}^n) = \begin{cases} \bb{Z} &\text{if } k = 0, n \\ 0 &\text{otherwise.}\end{cases}\]
\end{eg}


\appendix

\gdef\thesection{\Alph{section}} 
\makeatletter
\renewcommand\@seccntformat[1]{\appendixname\ \csname the#1\endcsname.\hspace{0.5em}}
\makeatother

\section{Proof of Theorem 7.11}\label{sec:proofs}

It remains to prove \Cref{thm:MH_normed_gps}, describing the iterated magnitude homology of a group with a conjugation-invariant norm. The proof makes use of three lemmas. Lemmas \ref{lem:cols_exact} and \ref{lem:row_col_0} are standard facts of homological algebra, for which proofs can be found in Appendix C of \cite{Roff2022} (Lemmas C.3.1 and C.3.2).

\begin{lem}\label{lem:cols_exact}
Let \(C\) be a first quadrant double complex with the property that every column is exact in vertical degrees greater than 0. Then
\[H_k \Tot (C) = H^h_k C_{\bullet0} \text{ for all }k \geq 0. \tag*{\qed}\]
\end{lem}

\begin{lem}\label{lem:row_col_0}
Let \(C\) be a first quadrant double complex such that \(C_{0p}=C_{p0} = 0\) for all \(p \in \bb{N}\). Then
\[H_k\Tot(C) \cong \begin{cases} 0 & k =0, 1 \\
H_1^h H_1^v C & k = 2.
\end{cases} \tag*{\qed}\]
\end{lem}

The third lemma is an elementary fact about bi-invariant metrics on groups.

\begin{lem}\label{lem:adj_split}
Let \(G\) be a group equipped with a bi-invariant metric \(d\). Elements \(g,h \in G\) are non-adjacent if and only if they can be factored as \(g = g_0g_1\) and \(h=h_0h_1\) in such a way that \(g_0 \neq h_0\), \(g_1 \neq h_1\), and
\[d(g,h) = d(g_0,h_0) + d(g_1,h_1).\]
\end{lem}

\begin{proof}
First, suppose \(g\) and \(h\) are non-adjacent: there exists some \(k \neq g, h\) such that \(d(g,h) = d(g,k) + d(k,h)\). We can factor \(g\) as \((gk^{-1})k\) and \(h\) as \(e h\); since \(k \neq g\), we know \(gk^{-1} \neq e\). Multiplying \(gk^{-1}\) and \(e\) on the right by \(k\) and using the right-invariance of the metric, we find that
\[d(gk^{-1}, e) + d(k,h) = d(g,k) + d(k,h) = d(g,h).\]
Conversely, suppose we have a factorization \(g = g_1g_0\) and \(h=h_1h_0\) with the relevant property. I claim that \(h_1g_0\) lies strictly between \(g\) and \(h\). Indeed,
\begin{align*}
d(g, h_1g_0) + d(h_1g_0, h) &= d(g_1g_0, h_1g_0) + d(h_1g_0, h_1h_0) \\
&= d(g_1, h_1) + d(g_0, h_0) \\
&= d(g,h)
\end{align*}
where the second equality uses the bi-invariance of the metric. This says that \(h_1g_0\) lies \emph{between} \(g\) and \(h\). By assumption, \(g_0 \neq h_0\) and thus \(h_1g_0 \neq h_1h_0 = h\). Similarly, \(h_1g_0 \neq g\), and hence \(h_1g_0\) lies \emph{strictly} between \(g\) and \(h\).
\end{proof}

With these, we can prove the theorem.

\begin{thmMH_normed_gps}
Let \(\cl{G} = (G, |-|)\) be a group with a conjugation-invariant norm. In length grading \(0\), iterated magnitude homology recovers the group homology of \(G\):
\[MH_\bullet^0(\cl{G}) \cong H_\bullet(G).\]
In length gradings \(\ell > 0\) we have \(MH_0^\ell(\cl{G}) = 0 = MH_1^\ell(\cl{G})\), while
\[MH_2^\ell(\cl{G}) \cong \bb{Z} \cdot \left\{ \begin{array}{l} \text{conjugacy classes of indecomposable} \\ \text{elements of norm \(\ell\) in \(G\)} \end{array} \right\}.\] 
\end{thmMH_normed_gps}

    \begin{lrbox}{\MOnePairA}%
        \begin{tikzcd}%
            \begin{bmatrix} g_0 \\ h_0 \end{bmatrix}
        \end{tikzcd}%
    \end{lrbox}

    \begin{lrbox}{\MOnePairB}%
        \begin{tikzcd}%
            \begin{bmatrix} g_0g_1 \\ h_0h_1 \end{bmatrix}
        \end{tikzcd}%
    \end{lrbox}

    \begin{lrbox}{\MOnePairC}%
        \begin{tikzcd}%
            \begin{bmatrix} g_1 \\ h_1 \end{bmatrix}
        \end{tikzcd}%
    \end{lrbox}

    \begin{lrbox}{\MTwoPairs}%
        \begin{tikzcd}%
            \begin{bmatrix} g_0 & g_1 \\ h_0 & h_1 \end{bmatrix}
        \end{tikzcd}%
    \end{lrbox}
    
    \begin{lrbox}{\MOneDoubleA}%
        \begin{tikzcd}%
            \begin{bmatrix} g_0 \\ h_0 \\ k_0\end{bmatrix}
        \end{tikzcd}%
    \end{lrbox}

    \begin{lrbox}{\MTwoDoubles}%
        \begin{tikzcd}%
            \begin{bmatrix} g_0 & g_1 \\ h_0 & h_1 \\ k_0 & k_1 \end{bmatrix}
        \end{tikzcd}%
    \end{lrbox}

\begin{proof}
Let \(C^* = C(MB^{MB^\Sigma}(\cl{G}))\) be the double complex associated to the double magnitude nerve of \(\cl{G}\). We will prove the theorem using the isomorphism
\[MH_\bullet^\ell(\cl{G}) \cong H_\bullet \Tot(C^\ell)\]
given, in each length grading \(\ell\), by the bisimplicial Eilenberg--Zilber Theorem.

In length grading \(0\), for every \(p,q \geq 0\) the group \(C^0_{pq}\) is freely generated by \((q+1)\times p\) matrices whose total length is 0. Every such matrix has the form
\[\begin{bmatrix}
g_0 & \cdots & g_n \\
g_0 & \cdots & g_n \\
\vdots & & \vdots \\
g_0 & \cdots & g_n
\end{bmatrix}\]
for some \(g_0, \ldots, g_n \in \cl{G}\), and thus is degenerate in the vertical simplicial structure described in \Cref{prop:double_MB_normed}. This tells us that every column of \(C^0\) is exact in vertical degrees greater than 0, so, by \Cref{lem:cols_exact}, \(MH^0_\bullet(\cl{G})\) is the homology of row 0. Inspecting the horizontal simplicial structure described in \Cref{prop:double_MB_normed}, one sees that this exactly the ordinary group homology of \(G\).

Now consider length grading \(\ell > 0\). Since the total length of the empty matrix is 0, the leftmost column of the double complex \(C^\ell\) vanishes, and as the total length of any \(1 \times p\) matrix is 0, the bottom row vanishes too. The southwest corner of \(C^\ell\) therefore has the form
\bdiag
& \vdots \arrow{d} & \vdots \arrow{d}\\
0  & \bb{Z} \cdot \left\{ \usebox{\MOneDoubleA} \right\} \arrow{d}{\partial^v_{11}} \arrow[swap]{l}{\partial^h_{02}} & \bb{Z} \cdot \left\{ \usebox{\MTwoDoubles} \right\} \arrow{d}{\partial^v_{21}} \arrow[swap]{l}{\partial^h_{12}} & \cdots \arrow{l} \\
0  & \bb{Z} \cdot \left\{ \usebox{\MOnePairA} \right\} \arrow{d}{\partial^v_{10}} \arrow[swap]{l}{\partial^h_{01}} & \bb{Z} \cdot \left\{ \usebox{\MTwoPairs} \right\} \arrow{d}{\partial^v_{20}} \arrow[swap]{l}{\partial^h_{11}} & \cdots \arrow{l} \\
0 & 0 & 0  & 
\ediag
where the differentials are induced by the bisimplicial structure described in \Cref{prop:double_MB_normed}. According to \Cref{lem:row_col_0}, then, we have
\[MH^\ell_k(\cl{G}) = H_k\Tot(C^\ell) \cong \begin{cases} 0 & k =0, 1 \\
H_1^h H_1^v (C^\ell) & k = 2.
\end{cases}\]
It remains to find \(MH_2^\ell(\cl{G})\) by computing \(H_1^h H_1^v (C^\ell)\).

Column 1 coincides with the magnitude chain complex in grading \(\ell\) of the metric space \(\cl{G}\), so---invoking \Cref{prop:MH1_met}, or directly---we find that \(H_1^v (C_{1\bullet}^\ell)\) is freely generated by pairs \((g,h)\) of adjacent elements such that \(d(g,h) = \ell\). Given \emph{any} pair \((g,h)\) such that \(d(g,h) = \ell\), we have
\[|gh^{-1}| = d(gh^{-1}, e) = d(g,h) = \ell,\]
and it is readily seen that \(g\) and \(h\) are adjacent if and only if \(gh^{-1}\) is indecomposable. Thus, there is a linear map
\[\phi: H_1^v (C_{1\bullet}^\ell) \to \bb{Z} \cdot \{\text{indecomposable elements of norm \(\ell\) in \(\cl{G}\)}\}\]
specified on generators by \(\phi(g,h) = gh^{-1}\). The map \(\phi\) is surjective (since \(g = \phi(g,e)\)) but not necessarily injective; we will show that on taking horizontal homology it descends to an isomorphism 
\[H_1^h H_1^v (C^\ell) \cong  \bb{Z} \cdot \{\text{conjugacy classes of indecomposable elements of norm \(\ell\) in \(\cl{G}\)}\}.\]

For this, consider the map \(\partial_{11}^h\). We have 
\[\partial_{11}^h \left( \usebox{\MTwoPairs} \right) = 0\]
unless \(d(g_0g_1, h_0h_1) = d(g_0,h_0) + d(g_1,h_1)\), in which case
\begin{align}
\partial_{11}^h \left( \usebox{\MTwoPairs} \right) = \begin{cases}
\usebox{\MOnePairA} - \usebox{\MOnePairB} & \text{if } g_1 = h_1 \\
\usebox{\MOnePairC} - \usebox{\MOnePairB} & \text{if } g_0 = h_0 \\
- \usebox{\MOnePairB} & \text{otherwise.}
\end{cases} \label{cond:norm}
\end{align}
(Notice that since \(d(g_0,h_0) + d(g_1,h_1) = \ell > 0\), if \(g_1 = h_1\) then \(g_0 \neq h_0\), and if \(g_0 = h_0\) then \(g_1 \neq h_1\).) By \Cref{lem:adj_split} a generator of \(C^\ell_{11}\) satisfies the third condition in (\ref{cond:norm}), and thus is in the image of \(\partial_{11}^h\), if and only if it is a non-adjacent pair; consequently, no generator of \(H_1^v (C_{1\bullet}^\ell)\) lies in the image of \(H^v(\partial_{11}^h)\). Taking the quotient of \(H_1^v (C_{1\bullet}^\ell)\) by the image of \(H^v(\partial_{11}^h)\) therefore has the effect of quotienting the generating set by two equivalence relations coming from the first two conditions in (\ref{cond:norm}):
\begin{itemize}
\item \((g_0, h_0) \sim_1 (g_0', h_0')\) if there exists \(k \in G\) such that \(g_0k = g_0'\) and \(h_0k = h_0'\);
\item \((g_1, h_1) \sim_2 (g_1', h_1')\) if there exists \(k \in G\) such that \(kg_1 = g_1'\) and \(kh_1 = h_1'\).
\end{itemize}

Suppose \(\phi(g_0,h_0) = \phi(g_1,h_1)\); that is, \(g_0h_0^{-1} = g_1h_1^{-1}\). Let \(k = h_0^{-1} h_1\). Then \(g_0 k = g_0 h_0^{-1} h_1 = g_1 h_1^{-1} h_1 = g_1\) and \(h_0 k = h_0 h_0^{-1} h_1 = h_1\),
so we have \((g_0,h_0) \sim_1 (g_1,h_1)\). Conversely, if \((g_0,h_0) \sim_1 (g_1,h_1)\) then
\[g_1 h_1^{-1} = (g_0 k) (h_0 k)^{-1} = g_0 k k^{-1} h_0^{-1} = g_0 h_0^{-1}.\]
This tells us that \(\phi\) descends to an isomorphism
\[\widetilde{\phi}: H_1^v (C_{1\bullet}^\ell) /{\sim_1} \xto{\cong} \bb{Z} \cdot \{\text{indecomposable elements of norm \(\ell\) in \(\cl{G}\)}\}.\]
Meanwhile, \((g_0,h_0) \sim_2 (g_1,h_1)\) if and only if there exists \(k\) such that
\[g_1 h_1^{-1} = (k g_0) (k h_0)^{-1} = k (g_0 h_0^{-1}) k^{-1},\]
that is, if and only if \(\widetilde{\phi}(g_0,h_0)\) and \(\widetilde{\phi}(g_1,h_1)\) are conjugate elements of \(G\). Thus, upon taking the quotient by \(\sim_1\) and \(\sim_2\) we obtain the claimed isomorphism
\[H_1^h H_1^v (C^\ell) \cong  \bb{Z} \cdot \left\{ \begin{array}{l} \text{conjugacy classes of indecomposable} \\ \text{elements of norm \(\ell\) in \(G\)} \end{array} \right\}.\qedhere\]
\end{proof}


\bibliographystyle{abbrvnat}

\bibliography{ERoff_Iterated_MH}

\end{document}